\documentclass[amssymb,amsfonts,reqno]{amsart}

\usepackage{amssymb,amsmath,amsthm,verbatim,IEEEtrantools, mathtools}

\usepackage{latexsym}

\usepackage{graphicx}

\def\be#1{\begin{equation}\label{#1}}
\def\bas{\begin{align*}}
\def\eas{\end{align*}}
\def\bi{\begin{itemize}}
\def\ei{\end{itemize}}

\usepackage{pdfpages}

\theoremstyle{plain}
   \newtheorem{theorem}[subsection]{Theorem}

   \newtheorem{proposition}[subsection]{Proposition}
   
   \newtheorem{lemma}[subsection]{Lemma}

\begin{document}

\author{Robert Fraser}
\address{Maxwell Institute of Mathematical Sciences and the School of Mathematics, University of
Edinburgh, JCMB, The King's Buildings, Peter Guthrie Tait Road, Edinburgh, EH9 3FD, Scotland}
\email{robert.fraser@@ed.ac.uk}

\author{James Wright}
\address{Maxwell Institute of Mathematical Sciences and the School of Mathematics, University of
Edinburgh, JCMB, The King's Buildings, Peter Guthrie Tait Road, Edinburgh, EH9 3FD, Scotland}
\email{J.R.Wright@@ed.ac.uk}




\title[On the local sum conjecture in two dimensions]{On the local sum conjecture in two dimensions}

\maketitle

\begin{abstract} In this paper we give an elementary proof of the local sum conjecture
in two dimensions. In a remarkable paper [CMN], this conjecture has been established
in all dimensions using sophisticated, powerful techniques from a research area
blending algebraic geometry with ideas from logic. The purpose of this paper is to give
an elementary proof of this conjecture which will be accessbile to a broad readership.
\end{abstract}

\section{Introduction}\label{introduction}

In their seminal paper \cite{DS}, Denef and Sperber formulated the following {\it local sum conjecture}.
Let $\phi \in {\mathbb Z}[X_1, \ldots, X_n]$ and consider the {\it local}  exponential sum
\[S_0 = S_0(\phi, p^s) \ := \  \frac{1}{p^{s n}} \sum_{\substack{x \in [\mathbb{Z}/p^s \mathbb{Z}]^n \\ x \equiv 0 \text{ (mod $p$) }}} e^{2 \pi i \phi(x)/p^s},\]
the local sum being a truncation of the complete exponential sum
\[S = S(\phi, p^s) \ := \  \frac{1}{p^{s n}} \sum_{x \in [\mathbb{Z}/p^s \mathbb{Z}]^n } e^{2 \pi i \phi(x)/p^s}\]
which selects the terms $x = (x_1, \ldots, x_n)$ where $ p \, |  x_j$ for all $1\le j \le n$.
The conjecture postulates that there exists a constant $C$, independent of $p$ and $s$, and a finite
set ${\mathcal P} = {\mathcal P}_{\phi}$ of primes such that for all $p \notin {\mathcal P}$,
\begin{equation}\label{DS-conjecture}
 |S_0| \ \le \ C s^{n-1} p^{-\sigma_c s} 
\end{equation}
where $\sigma_c = \sigma_c(\phi)$ is the complex oscillation index\footnote{In the literature, oscillation indices tend to be
defined as negative numbers. We will consider their absolute values and define them as positive numbers.} of $\phi$ at $0$ (See \cite{AGV}, section 13.1.5).
We will recall the precise definitions for this and other notions in the next section. The conjecture \eqref{DS-conjecture}  is related
to one of the Igusa conjectures on exponential sums which posits  similar uniform bounds for $S$ when $\phi$ is
any homogeneous polynomial.

In \cite{DS}, Denef and Sperber proved \eqref{DS-conjecture} when\footnote{
In \cite{DS} a couple of minor auxilary conditions were also imposed.} $\phi$ is ${\mathbb C}$-nondegenerate,
a notion introduced in \cite{Kushnirenko} and coined as {\it nondegenerate with respect to its Newton diagram} in \cite{DS}. In this case the complex oscillation
index $ \sigma_c(\phi) = 1/d(\phi)$ is the reciprocal of the {\it Newton distance}\footnote{
Strictly speaking, when the dimension is large, this is only true if $d(\phi) > 1$; see \cite{AGV}.} $d(\phi)$ of $\phi$, see
\cite{Varchenko}. In the same paper, Denef and Sperber also
establish the Igusa conjecture under the same hypothesis, when $\phi$ is ${\mathbb C}$-nondegenerate.

These conjectures are motivated in part by the striking, well-known similarities between bounds for $S_0$ (and $S$)
and bounds for euclidean oscillatory integrals; for a real-valued $\phi \in C^{\infty}({\mathbb R}^n)$
such that $\nabla \phi(0) = 0$, set
\[ I = I_{\psi}(\phi, \lambda) \ := \ \int_{{\mathbb R}^n} e^{2\pi i \lambda \phi(x)} \psi(x) \, dx \]
where $\lambda \in {\mathbb R}$ is a large real parameter and the amplitude $\psi \in C^{\infty}_c({\mathbb R}^n)$
is supported in a neighbourhood of $0$. We are interested in those exponents 
$\beta = \beta(\phi) \ge 0$ such that the bound
\begin{equation}\label{I-bound}
|I_{\psi}(\phi, \lambda)| \ \le \ C_{\phi} \, |\lambda|^{-\beta}
\end{equation}
holds for all $|\lambda| \ge 1$ and for all $\psi\in C^{\infty}_c({\mathbb R}^n)$ supported in some
neighbourhood of $0$. The (real) {\it oscillation index} $\sigma_r(\phi)$ is defined as the supremum of 
those $\beta$ such that \eqref{I-bound} holds. It is
known that $\sigma_c(\phi) \le \sigma_r(\phi)$ and examples show that strict inequality can occur (see \cite{AGV}, Lemma 13.6]).
One could formulate a stronger conjecture by replacing the exponent $\sigma_c(\phi)$ in \eqref{DS-conjecture}
with $\sigma_r(\phi)$.

In the literature, the various oscillation indices (for example, $\sigma_r(\phi)$ and $\sigma_c(\phi)$) and the
height function $h(\phi)$ introduced below are defined with respect to a critical point of $\phi$ and so we imposed
above the condition $\nabla\phi(0) = 0$.
When $\nabla \phi(0) \not= 0$, then \eqref{I-bound} holds for every $\beta > 0$ if $\psi$ has sufficiently
small support. Hence $I = I_{\psi}(\phi, \lambda)$ decays rapidly in $\lambda$. The corresponding case for the exponential
sum $S_0(\phi, p^s)$, when $\nabla \phi(0,0) \not= 0$, can be easily analysed. Say
$\partial_x \phi(0,0) = c \not= 0$ and by enlarging the set ${\mathcal P}$ to include the prime divisors
of $c$ if needed, we may suppose that  
$p \not| \, c$ for all $p \notin {\mathcal P}$. In this case, we have $S_0(\phi, p^s) = 0$ when $s\ge 2$ (this is
the analogue of rapid decay in the local field setting) and
when $s = 1$, we have $S_0(\phi, p) = p^{-2}$. In section \ref{basic-bound} we will give a proof 
of these simple facts. 

Henceforth we will always assume that our polynomial $\phi$ satisfies
$\nabla \phi(0,0) = 0$. 

These oscillation indices, $\sigma_c(\phi)$ and $\sigma_r(\phi)$, do not depend on the underlying coordinate system;
changing coordinates does not affect the decay rates in \eqref{I-bound} as we require the bound to hold for
all smooth $\psi$, supported in a sufficiently small neighbourhood of $0$. Unfortunately the Newton distance $d(\phi)$,
a quantity we can easily compute,
does depend on the coordinate system. Nevertheless, in two dimensions, we can get our hands on the
oscillation index $\sigma_r(\phi)$ since it is known that $\sigma_r(\phi) = 1/h(\phi)$ where
$h(\phi) := \sup_z d_z(\phi)$ is the so-called {\it height} of $\phi$. Here the supremum is taken over all
local coordinate systems $z= (x,y)$ of the origin (real-analytic coordinate systems if the
phase $\phi$ is real-analytic and smooth coordinate systems if $\phi$ is smooth) and $d_z(\phi)$ denotes
the Newton distance of $\phi$ in the coordinates $z$. 

In two dimemsions, the supremum $\sup_z d_z(\phi) = d_{z_0}(\phi)$ is attained in the definition of the height $h(\phi)$
(any such coordinate system $z_0$ is called {\it adapted})
and for any smooth, real-valued $\phi$ of finite-type,
\begin{equation}\label{I-bound-sharp}
 |I_{\psi}(\phi, \lambda)| \ \le \ C \, \log^{\nu}(|\lambda|) |\lambda|^{-1/h(\phi)}
\end{equation}
for large $\lambda$ and all $\psi$ supported in a sufficiently small neighbourhood of $0$. 
Here $\nu(\phi) \in \{0,1\}$ is the so-called {\it Varchenko's exponent} (also known as the
multiplicity of the oscillation index). Furthermore, 
\begin{equation}\label{I-bound-lower}
\lim_{\lambda \to \infty} \frac{\lambda^{1/h}}{\log^{\nu}(\lambda)} \,  I_{\psi}(\phi, \lambda) \ = \
c \, \psi(0) 
\end{equation}
where $c = c_{\phi}$ is nonzero.\footnote{The existence of this limit is proved under the additional
condition that the principal face of $\phi$ in adapted coordinates is compact.} In this generality,
the results in \eqref{I-bound-sharp} and \eqref{I-bound-lower} were established by Ikromov and M\"uller
in \cite{IM-adapted} and \cite{IM-uniform}. Their work was influential in our analysis establishing the following.

\begin{theorem}\label{main} Let $\phi \in {\mathbb Z}|X,Y]$ such that $\nabla\phi(0,0) = 0$. There there exists a finite set
${\mathcal P}$ of primes and a constant $C = C_{\phi}$ such that for any $p \notin {\mathcal P}$ and $s\ge 1$,
\begin{equation}\label{S_0-main-estimate}
|S_0(\phi, p^s)| \ \le \ C \ s^{\nu(\phi)} p^{-s/h(\phi)}
\end{equation}
holds for all $\phi$ except for an exceptional class ${\mathcal E}$. For $\phi \in {\mathcal E}$,
the estimate \eqref{S_0-main-estimate} holds with $\nu = 1$; that is $|S_0(\phi, p^s)| \le C s p^{-s/h(\phi)}$
holds for $\phi \in {\mathcal E}$.
\end{theorem}
The class ${\mathcal E}$ consists of those polynomials of the form
\[\phi(x,y) \ = \  \ \ a(b y^2 + cxy + d x^2)^m \ + \ {\rm higher \ order \ terms}\footnote{We will describe this class
precisely in section \ref{definitions}.}\]
where the quadratic polynomial $b y^2 + c xy + d x^2$ is
irreducible over the rationals ${\mathbb Q}$. For example when $\phi(x,y) = a (x^2 + y^2)^m$,
we have $h(\phi) = m$ and $\nu(\phi) = 0$. However when $m\ge 2$ and $p \equiv 1$ mod $4$, then
$|S_0(\phi, p^s)| \sim  s p^{-s/m}$ for infinitely many $s\ge 1$. Furthermore when $p \equiv 3$ mod $4$,
then $|S_0(\phi, p^s)| \sim p^{-s/m}$ for infinitely many $s\ge 1$. These calculations are not difficult; see for example
\cite{W-igusa} where more general bounds are derived.



 As mentioned in the abstract, Cluckers, Mustata and Nguyen \cite{CMN} established the local sum conjecture
\eqref{DS-conjecture} in all dimensions and much more; they also establish the Igusa conjecture
for complete exponential sums $S(\phi, p^s)$ where $\phi$ is a general homogeneous polynomial. The
estimate \eqref{S_0-main-estimate} in Theorem \ref{main} is a slight strengthening in the two dimensional
case but more importantly, we establish \eqref{S_0-main-estimate} using elementary arguments, only 
basic $p$-adic analysis is used. A key step in our argument will follow ideas from Ikromov and M\"uller
in \cite{IM-adapted} in the euclidean setting which in turn were inspired from the arguments developed
in \cite{PSS} which gave an elementary treatment of Karpushkin's work \cite{Karpushkin} on stability bounds of euclidean oscillatory integral
estimates in two dimensions. 

The main effort in this paper is to rework euclidean arguments in the local field setting.
Basic euclidean arguments for estimating oscillatory integrals rely
heavily on the order structure of the reals (in applications of the mean value and intermediate value theorems
which are implicitly used in integration by parts arguments). We need to readdress these arguments, relying
more on rudimentary sublevel set estimates (bounds for the number of solutions to polynomial congruences)
in place of integration by parts arguments. These sublevel set bounds will be derived from a higher order
Hensel lemma and so matters are kept on an elementary level.

\subsection*{Notation} All constants $C, c,  c_0 > 0$ throughout this paper
will depend only on the  polynomial $\phi$, although
the values of these constants may change from line to line. Often it will be
convenient to suppress explicitly mentioning the constants $C$ or $c$ in these inequalities
and we will use the notation $A\lesssim B$ between positive quantities $A$ and $B$ to
denote the inequality $A\le C B$ (we will also denote this as $A = O(B)$). When we want to
emphasise the dependence of the implicit constant in $A\lesssim B$ on a parameter $k$,
we write $A\lesssim_k B$ to denote $A \le C_k B$. Finally we use the notation $A \sim B$ to
denote that both inequalities $A\lesssim B$ and $B\lesssim A$ hold.

\subsection*{Acknowledgements} We would like to thank Allan Greenleaf and Malabika Pramanik
for informative and enlightening discussions about oscillatory integrals.

This material is based upon work supported by the National Science Foundation under Award No. 1803086.
\section{Outline of the paper}\label{outline}
In the next section we will define precisely the various notions introduced above, including
reviewing the Newton polyhedron, diagram and distance of a polynomial. We will also give a quick
review of the required $p$-adic analysis that we will use and show how we can lift our exponential
sum $S_0$ over ${\mathbb Z}/p^s {\mathbb Z}$ to an oscillatory integral over the $p$-adic field. 
This will illustrate the close analogy between these kinds of exponential sums and euclidean oscillatory integrals.
In Sections \ref{basic-bound} and \ref{proof-basic} we will derive a basic bound for $S_0(\phi, p^s)$ which will imply \eqref{S_0-main-estimate}  
in Theorem \ref{main} when the coordinates $z = (x,y)$ of our given polynomial $\phi(x,y)$ are
{\it adapted}. 

This basic bound will employ a useful estimate for exponential sums in one variable which depends on
a generalisation of the classical Hensel lemma. We will outline the proof of this one dimensional bound in Section \ref{hensel}.

The main effort then will be to find a change of variables to put our polynomial $\phi$
into adapted coordinates. In general the change of variables that accomplishes this will be analytic.
Attempting to keep our analysis on an elementary level, we will find a polynomial
change of variables 
\[p(x,y) \ = \ (p_1(x,y), p_2(x,y)) \in  {\mathbb Q}[X,Y]\] so that the new phase ${\tilde \phi}(x,y) = \phi(p(x,y))$
will be a polynomial with rational coefficients. The polynomial ${\tilde \phi}(x,y)$ will not quite be in
adapted coordinates but nevertheless the bound established in Sections \ref{basic-bound} and \ref{proof-basic} will be sufficient to
prove Theorem \ref{main}.

To produce this change of variables, we will follow an algorithm due to Ikromov and M\"uller \cite{IM-adapted} 
in the euclidean setting. They, in turn, blend ideas from two different algorithms due to Varchenko \cite{Var-adapted} and
Phong, Stein and Sturm \cite{PSS}. This will be carried out in Section \ref{sect-CoV}. The algorithm producing
this change of variables with rational coefficients is based on the clustering of the roots of $\phi$ which
can be expressed in terms of Puiseux series.

\section{Definitions and preliminaries}\label{definitions}

A good reference for the  following basic results and definitions regarding oscillatory integrals
can be found in \cite{AGV}. 

\subsection*{The oscillation indices} Any polynomial $\phi \in {\mathbb Z}[Z_1, \ldots, X_n]$ can be viewed as a real-valued
phase and so the oscillation indices discussed in the introduction make sense for $\phi$. 
 The {\it complex oscillation index} is defined
as the supremum of $\beta$'s where the bounds $|\int_{\Gamma} e^{2 \pi i\lambda \phi(x)} dx| \le C_{\Gamma} \lambda^{-\beta}$
hold for large $\lambda > 1$ and all $n$-dimensional chains $\Gamma$ 
in a sufficiently small neighbourhood of $0$ in ${\mathbb C}^n$,
such that the imaginary part $\phi$ is strictly positive on the boundary of $\Gamma$. The 
complex oscillation index $\sigma_c(\phi)$ is smaller (and can be strictly smaller) than the
{\it oscillation index} $\sigma_r(\phi)$ of $\phi$ defined in the introduction. In general these
indices are difficult to compute. However when $\phi$ satisfies a certain nondegeneracy condition,
then we can get our hands on these numbers.

\subsection*{The Newton polyhedron and diagram}
To describe this nondegeneracy condition, we need to recall the definition of the Newton
polyhedron of a polynomial $\phi$; we will restrict ourselves to two dimensions although these
notions make sense in any dimension. Let ${\mathbb N} := \{0, 1, 2, \ldots \}$ include zero.
For any polynomial $\phi(x,y) =  \sum_{j,k} c_{j,k} x^j y^k$, we call the set 
${\mathcal S}(\phi) := \{(j,k)\in {\Bbb N}^2 \setminus \{0\} : c_{j,k} \not= 0 \}$, the {\it reduced} support of $\phi$. 
The {\it Newton polyhedron} $\Gamma(\phi)$ of $\phi$ is the convex hull of the
union of all quadrants $(j,k) + {\Bbb R}^2_{+}$ in ${\Bbb R}^2$
with $(j,k)\in {\mathcal S}(\phi)$. Let $\Delta(\phi)$ be the collection of compact faces (vertices and edges)
of $\Gamma(\phi)$. 
The {\it Newton diagram} ${\mathcal N}_d(\phi)$ is the union of the faces in $\Delta(\phi)$.

For each face $\gamma$ of $\Gamma(\phi)$, we set $\phi_{\gamma}(x,y) = \sum_{(j,k) \in \gamma} c_{j,k} x^j y^k$.
We say that $\phi$ is {\it ${\mathbb C}$-nondegenerate} ({\it ${\mathbb R}$-nondegenerate}) if for every
compact face $\tau \in \Delta(\phi)$,
$$
\nabla \phi_{\tau}(x,y) \ = \ (\frac{\partial \phi_{\tau}}{\partial x} (x,y),  \
\frac{\partial \phi_{\tau}}{\partial y} (x,y))
$$
never vanishes in $({\mathbb C}\setminus \{0\})^2$ ($({\mathbb R}\setminus \{0\})^2$).

\subsection*{The Newton distance and the height function}
If we use coordinates $(t_1,t_2)$
for points in the plane containing the Newton polyhedron, consider
the point $(d,d)$ in this plane where the bisectrix $t_1 = t_2$
intersects the boundary of $\Gamma(\phi)$. The coordinate
$d = d(\phi)$ is called the {\it Newton distance} of $\phi$ in the
coordinates $z = (x,y)$. The {\it principal face} $\pi(\phi)$ is the face
of minimal dimension (an edge or vertex) which contains the point $(d,d)$. 
Following \cite{IM-adapted}, we call $\phi_{\pi(\phi)}$ the {\it principal part} of $\phi$
and denote it by $\phi_{{\text pr}}$.

When $\phi$ is ${\mathbb R}$-nondegenerate, then the
oscillation index is the reciprocal of the Newton distance;\footnote{This is
true in two dimensions but we need to assume in addition that $d(\phi) > 1$ in higher dimensions.}  
$\sigma_r(\phi) = 1/d(\phi)$. In two dimensions,
we can still get our hands on the elusive oscillation index $\sigma_r(\phi)$ for general $\phi$ since 
$\sigma_r(\phi) = 1/h(\phi)$ is the reciprocal of the height $h(\phi) := \sup_z d_z$ where 
$d_z$ is the Newton distance of $\phi$ in the coordinates $z = (x,y)$.
Furthermore the supremum is attained $h(\phi) = d_{z_0}$ and we call any such coordinate system
$z_0$ {\it adapted}. This is no longer the case in higher dimensions. 

The notions of Newton polyhedron $\Gamma(\phi)$, Newton diagram ${\mathcal N}_d(\phi)$,
Newton distance $d(\phi)$ as well as principal face $\pi(\phi)$ and principal
part $\phi_{{\text pr}}$ easily extend from polynomials to any real-analytic function. This
will be useful in Section \ref{sect-CoV}.

\subsection*{The Varchenko exponent}
The {\it Varchenko exponent} $\nu(\phi)$ was introduced in \cite{Var-adapted} and is defined to be zero unless $h(\phi) \ge 2$ and in this case,
when the principal face $\pi(\phi^z)$ of $\phi^z$ in an adapted coordinate system $z$
is a vertex, we define $\nu(\phi)$ to be $1$. Otherwise we set $\nu(\phi) =0$.

\subsection*{The exceptional class ${\mathcal E}$} With the notions of the Newton diagram and the
principal part of $\phi$, we can now describe the exceptional class ${\mathcal E}$ precisely. It is the
class of polynomials $\phi$ whose principal part $\phi_{\rm pr}(x,y) = a (b x^2 + c xy + d y^2)^m$
where the quadratic polynomial $b x^2 + c xy + d y^2$ is irreducible over the rationals
${\mathbb Q}$. 

\subsection*{The $p$-adic number field}
We fix a prime $p$ and define the $p$-adic absolute value\footnote{We will also
use the notation $|z|$ for the usual absolute value on elements $z \in {\mathbb C}$ but the context
will make it clear which absolute value is being used.}  $|\cdot| = |\cdot|_p$ on the field
of rationals ${\mathbb Q}$ as follows.
For integers $a \in {\mathbb Z}$, we define $|a| := p^{-k}$ where $k\ge 0$ is the largest power
such that $p^k$ divides $a$. This $p$-adic absolute value extends to all rationals $a/b$ by $|a/b| = |a|/|b|$
and satisfies the basic conditions $|u v| = |u| |v|$ and $|u + v| \le |u| + |v|$ for all rationals $u,v \in {\mathbb Q}$,
giving ${\mathbb Q}$ a metric space structure $d(u,v) = |u - v|$.
The $p$-adic absolute value in fact satisfies a stronger version of triangle inequality called the ultrametric inequality: $|u+v| \le \max(|u|, |v|)$. This implies
$|u+v| = |u|$ if $|v| < |u|$ and so if $v \in B_r(u) := \{ w \in {\mathbb Q}: |w - u| \le r\}$, then $B_r(v) = B_r(u)$.

The $p$-adic field ${\mathbb Q}_p$ is the completion of the rational field ${\mathbb Q}$ with respect to the metric
defined by the $p$-adic absolute value. The elements in the completed field $x \in {\mathbb Q}_p$ can be 
represented by a Laurent series 
\begin{equation}\label{series-rep}
x \ = \ \sum_{j=-N}^{\infty} a_j p^j, \ \ \ \ a_j \in {\mathbb Z}/p{\mathbb Z} \ = \ \{0,1,\cdots, p-1\},
\end{equation}
convergent with respect to $|\cdot| = |\cdot|_p$ which extends uniquely to all of ${\mathbb Q}_p$
by $|x| = p^N$ where $a_{-N} \not= 0$ is the first term of the series representation \eqref{series-rep}.
We also define $|0| = 0$.

The compact unit ball $B_1(0) = \{ x \in {\mathbb Q}_p : |x| \le 1 \}$ plays a special role as it is a
ring due to the ultrametric inequality. We call this compact ring the {\it ring of $p$-adic integers}
and denote it by ${\mathbb Z}_p$. Hence ${\mathbb Q}_p$ is a locally compact abelian group and
has a unique Haar measure $\mu$ which we normalise so that $\mu({\mathbb Z}_p) = 1$. To carry out
Fourier analysis on ${\mathbb Q}_p$, we fix a non-principal additive character ${\rm e}$
defined by \[{\rm e}(x) \ := \  e^{2\pi i [\sum_{j=-N}^{-1} a_j p^j] } \ \ {\rm where} \ x \ {\rm
is \ represented \ as \ in \ \eqref{series-rep}}.\] 
All other characters $\chi$ on ${\mathbb Q}_p$ are given by $\chi(x) = {\rm e}(v x)$ for some $v\in {\mathbb Q}_p$.
Hence the Fourier dual of ${\mathbb Q}_p$ is itself. 

\subsection*{Hensel's lemma} The following basic lemma harks back to the origins of $p$-adic analysis
and it, together with a generalisation described in Section \ref{hensel}, will be useful for us.

\begin{lemma}\label{hensel-lemma} Let $g \in {\mathbb Z}[X]$ such that $g(x_0) \equiv 0$ mod $p^s$ for some
integer $x_0$. If $p^{\delta} || g'(x_0)$ (or $|g'(x_0)| = p^{-\delta}$) 
for some $\delta < s/2$, then there exists a unique $x \in {\mathbb Z}_p$
such that $g(x) = 0$ and $x \equiv x_0$ mod $p^{s-\delta}$. 
\end{lemma}

For a proof of Hensel's lemma, see \cite{Robert}, Chapter 1.6.

\subsection*{The sum ${S}_0(\phi, p^s)$ as an oscillatory integral}\label{sum-osc}
It is natural to analyse ${S}_0(\phi, p^s)$ by lifting this sum to an oscillatory integral defined
over the $p$-adic field ${\mathbb Q}_p$.

First let us see how the complete exponential sum $S(\phi, p^s)$ can be written
as the following oscillatory integral; we claim that
\begin{equation}\label{p-osc-int}
{S}(\phi, p^s) \ = \  \ \int\!\!\!\int_{{\mathbb Z}_p \times {\mathbb Z}_p} \, {\rm e}(p^{-s} \phi(x,y)) \, d\mu(x)
d\mu(y)
\end{equation}
holds. 
Consider a pair $x_0, y_0$ of integers and note that for any
$x \in B_{p^{-s}}(x_0)$
and $y \in B_{p^{-s}}(y_0)$,
we have ${\rm e}(p^{-s} \phi(x,y)) = {\rm e}(p^{-s} \phi(x_0, y_0))$. This simply follows from the definition of 
the character ${\rm e}$. Hence the oscillatory integral in \eqref{p-osc-int} can be written as
\[
\sum_{(x_0, y_0) \in [{\mathbb Z}/p^s{\mathbb Z}]^2} 
 \int\!\!\!\int_{B_{p^{-s}}(x_0)\times B_{p^{-s}}(y_0)}  \, {\rm e}(p^{-s} \phi(x,y)) \, d\mu(x)
d\mu(y)
\]
\[
= \ \sum_{(x_0, y_0) \in [{\mathbb Z}/p^s{\mathbb Z}]^2}  {\rm e}(p^{-s} \phi(x_0,y_0)) \mu(B_{p^{-s}})^2 \ = \
p^{-2s} \sum_{(x_0, y_0) \in [{\mathbb Z}/p^s{\mathbb Z}]^2}  e^{2\pi i\phi(x_0,y_0)/p^s}
\] 
and this last sum is our complete exponential sum $S(\phi, p^s)$. The last equality follows
since ${\rm e}(p^{-s} \phi(x_0,y_0)) =  e^{2\pi i\phi(x_0,y_0)/p^s}$ by the definition of ${\rm e}$.

A similar argument shows that our local sum $S_0(\phi, p^s)$ has the following
oscillatory integral representation:
\begin{equation}\label{S_0-osc-rep}
S_0(\phi, p^s) \ = \ 
\iint_{|x|, |y| \leq p^{-1}} {\rm e} (p^{-s} \phi(x,y)) \, d\mu(x)  d\mu(y).
\end{equation}
To simplify notation, we will denote the Haar measure $d\mu(x)$ by $dx$ and $\mu(E)$ by $|E|$.

\section{A reduction of Theorem \ref{main} to a basic bound for $S_0(\phi, p^s)$}\label{basic-bound}

In this section we will give a basic bound on the oscillatory integral in \eqref{S_0-osc-rep} which represents
our local sum $S_0(\phi, p^s)$. This bound by itself will fall short in proving the desired bound
\eqref{S_0-main-estimate} in Theorem \ref{main} and so one of our main tasks will be to
find a change of variables in \eqref{S_0-osc-rep} so that the bound formulated
in this section, with the transformed phase under this change of variables, 
is sufficient to establish Theorem \ref{main}.   

First though, we establish the simple facts about the exponential sum $S_0(\phi, p^s)$ when $\nabla\phi(0,0) \not= 0$
mentioned in the Introduction, allowing us to reduce matters to the case when $\nabla\phi(0,0) = 0$. 
Without loss of generality suppose $\phi(0,0) = 0$ and $\partial_x \phi(0,0) = c \not= 0$ and $p \not| \, c$ whenever $p \notin {\mathcal P}$.
Then for any integer $y \equiv 0$ mod $p$, consider the polynomial $g \in {\mathbb Z}[X]$ defined by $g(x) = \phi(x,y)$
and note that $p \not| \, g'(x)$ for every $x \equiv 0$ mod $p$. Hence by Hensel's lemma, the map $x \to g(x)$
defines a bijection on $\{ x \in {\mathbb Z}/p^s{\mathbb Z} : x \equiv 0 \, {\rm mod} \ p \}$ so that
\[\sum_{\substack{ x \in \mathbb{Z}/p^s \mathbb{Z}  \\x \equiv 0 \text{ (mod $p$)}}} e^{2 \pi i g(x)/p^s}
\ = \ \sum_{\substack{ u = 0  \\ p \, | \, u }}^{p^s - 1} e^{2 \pi i u/p^s}\]
which is equal to zero when $s\ge 2$, and equal to 1 when $s=1$. Hence
\[S_0(\phi, p^s) \ = \   \frac{1}{p^{2 s}} \sum_{\substack{(x,y) \in [\mathbb{Z}/p^s \mathbb{Z}]^2 \\ x,y \equiv 0 \text{ (mod $p$) }}} e^{2 \pi i \phi(x,y)/p^s} \ = \ 0\]
when $s\ge 2$ and equal to $p^{-2}$ when $s=1$.

A key result in this paper is the following. 

\begin{theorem}\label{CoV} For any $\phi \in {\mathbb Z}[X,Y]$ with $\nabla \phi(0,0) \neq 0$, we can find a polynomial
$\psi \in {\mathbb Q}[X]$ such that if ${\tilde \phi}(x, y) = \phi(x, y + \psi(x))$, then
$h(\phi) = h({\tilde \phi}) =  h({\tilde \phi}_{\rm pr})$.
\end{theorem} 

This result was established in the euclidean setting by Ikromov and M\"uller, \cite{IM-adapted}.
We follow their argument closely but with an extra effort to ensure that the polynomial
$\psi$ we end up with has rational coefficients. We postpone the proof until Section \ref{sect-CoV}.

This change of variables $(x,y) \to (x, y + \psi(x))$ depends only on $\phi$. We will include in our
exceptional set ${\mathcal P}$ those primes which arise as prime divisors of the coefficients of 
the transformed phase ${\tilde \phi}(x, y) = \phi(x,y + \psi(x)) \in {\mathbb Q}[X,Y]$. Hence
${\tilde \phi}$ will not only be a polynomial
with rational coefficients, these coefficients will be units (their $p$-adic absolute values are equal to 1)
in the ring of $p$-adic integers ${\mathbb Z}_p$. Hence implementing this change of variables in
\eqref{S_0-osc-rep} shows that $S_0(\phi, p^s) = S_0({\tilde \phi}, p^s)$.

Therefore in order to prove Theorem \ref{main}, according to Theorem \ref{CoV}, 
 it suffices to assume that
the phase $\phi$ in the oscillatory integral \eqref{S_0-osc-rep} representing the local sum $S_0$
is a polynomial with rational coefficients lying in ${\mathbb Z}_p$ and $h(\phi) = h(\phi_{\rm pr})$. 

We  begin with the decomposition 
\begin{equation}\label{S_0-decomp}
S_0(\phi,p^s) \ = \ \sum_{\tau \in \Delta(\phi)} \sum_{\substack{\vec l \in \mathbb{N}^2 \\ F(\vec l) = \tau}} 
\iint_{|x|=p^{-l_1}, |y| = p^{-l_2}} {\rm e}(p^{-s } \phi(x,y) \, dx dy
\end{equation}
introduced in \cite{DS}. Here, for each ${\vec l} = (l_1, l_2) \in {\mathbb N}^2$, 
$F(\vec l)$ is the face of $\Gamma(\phi)$ of largest dimension which is contained in the 
supporting line of $\Gamma(\phi)$ perpendicular to ${\vec l}$. In other words,
\[F(\vec l) \ = \ \{ {\vec t}\in \Gamma(\phi): \, {\vec t} \cdot {\vec l} = N(\vec l) \} \ \ {\rm where} \ \ 
N(\vec l) \ := \ \min_{{\vec t} \in \Gamma(\phi)}  {\vec t}\cdot {\vec l}.\]
Note that $F(\vec l)$ is a compact face of $\Gamma(\phi)$ if and only if
$l_1 l_2 \not= 0$ which explains why only compact faces $\Delta(\phi)$ enter into the decomposition of $S_0$ above.

For each compact $\tau \in \Delta(\phi)$ and each $\vec l$ such that $\tau = F(\vec l)$,
write $\phi(x,y) = \phi_{\tau}(x,y) + h_{\tau}(x,y)$. Then
\[ \phi_{\tau}(p^{l_1}x, p^{l_2} y) \ = \ \sum_{(j,k)\in \tau}  p^{(j,k)\cdot {\vec l}} \, c_{j,k}  x^j y^k \ = \
p^{N(\vec l)} \phi_{\tau}(x,y) \]
and 
\[ h_{\tau}(p^{l_1}x, p^{l_2} y) \ = \ \sum_{(j,k)\in \Gamma(\phi)\setminus{\tau}}  p^{(j,k)\cdot {\vec l}} \,
c_{j,k} x^j y^k \ = \
p^{N(\vec l)} \,  p \ g_{\tau}(x,y) \]
for some polynomial $g \in {\mathbb Z}_p[X,Y]$. 

Changing variables to normalise the region of integration, we have
\[S_0(\phi,p^s) = \sum_{\tau \in \Delta(\phi)} \sum_{\substack{\vec l \in \mathbb{N}^2 \\ F(\vec l) = \tau}} p^{-l_1 - l_2} 
\iint_{|x|, |y| = 1} {\rm e}(p^{-s + N(\vec l)} (\phi_{\tau} (x,y) + p g_{\tau}(x,y))) \, dx dy .\]
Now let us fix a compact face $\tau \in \Delta(\phi)$. If $\tau = \{(\alpha, \beta)\}$ is a vertex, then
$\phi_{\tau}(x,y) = c x^{\alpha} y^{\beta}$ is a monomial where $c$ is a rational number with $|c| = 1$.
If $\tau$ is a compact edge, then 
\[ \tau \ \subset \ \{(t_1, t_2) :  q t_1 + m t_2 = n \}\]
for some integers $(m,q,n) = (m_{\tau}, q_{\tau}, n_{\tau})$ with ${\rm gcd}(m,q) = 1$
and $\phi_{\tau}$ is a quasi-homogeneous polynomial, $\phi_{\tau} (r^{\kappa_1} x, r^{\kappa_2} y) = r \phi_{\tau}(x,y)$
for $r>0$ where $\kappa_1 = q/n$ and $\kappa_2 = m/n$. The polynomial $\phi_{\tau}$ consists of at least two terms and so
by homogeneity, we can factor\footnote{See \cite{IM-adapted} for details.}
\begin{equation}\label{phi-factored}
\phi_{\tau}(x,y) \ = \ c \, x^{\alpha} y^{\beta} \prod_{j=1}^N (y^q - \xi_j x^m)^{n_j}
\end{equation}
for some roots $\{\xi_j\}_{j=1}^N$ lying in $\mathbb{Q}^{\text{alg}}$. It will be convenient for us to think
of the roots $\{\xi_j\}$ as lying in some finite field extension of ${\mathbb Q}_p$. Again $c$ is a rational with $|c| = 1$.

Assume, without loss of generality, that $1 \leq q \leq m$. The exponent $d_{\tau} = d_{\tau}(\phi)$, where
\begin{equation}\label{homo-degree}
d_{\tau} \ := \ \frac{1}{\kappa_1 + \kappa_2} \ = \ \frac{n}{m+q} \ = \ \frac{q \alpha + m \beta + q m M}{m+q}
\end{equation}
and $M := \sum_{j=1}^N n_j$, is called the {\it homogeneous distance}\footnote{We are borrowing terminology
from \cite{IM-adapted}.} of $\phi_{\tau}$. The point $(d_{\tau}, d_{\tau})$ on the bisectrix lies on the line
$\{(t_1, t_2) : q t_1 + m t_2 = n\}$ containing $\tau$. Hence $d_{\tau}(\phi) \le d(\phi)$ for every compact
edge $\tau \in \Delta(\phi)$. If $\tau$ is the principal face, then $d_{\tau}(\phi) = d(\phi)$
is the Newton distance of $\phi$. 

The following simple lemma will be useful.

\begin{lemma}\label{simple-lemma} Let $\tau \in \Delta(\phi)$ be a compact edge. 

(a) If $\tau$ is not the principal face, then $M = \sum_{j=1}^N n_j \le  d_{\tau}$. Furthermore
strict inequality $M < d_{\tau}$ holds unless an endpoint of $\tau$ lies on the bisectrix.

(b) Suppose that $\tau = \pi(\phi)$ is the principal face and $n_{j_{*}} := \max_{1\le j\le N} n_j > d(\phi)$.
Then $n_j < d(\phi)$ for all $j \not= j_{*}$ and $\xi_{j_{*}} \in {\mathbb Q}$.

(c) Again suppose that $\tau = \pi(\phi)$ but now
$n_{j_{*}} = d(\phi)$. Then either $\phi_{\rm pr}(x,y) = c (y^2 + b xy + d x^2)^n$ 
is a power of a quadratic form or 
\begin{equation}\label{q=1}\phi_{\rm pr}(x,y) \ = \  c \,  x^{\alpha} y^{\beta} \prod_{j=1}^N (y - \xi_j x^m)^{n_j}, 
\end{equation}
$n_j < n_{j_{*}}$ for all $j\not= j_{*}$ and $\xi_{j_{*}} \in {\mathbb Q}$.

\end{lemma}

\begin{proof} To prove (a), let us suppose that $\tau$ lies below the bisectrix so that the left endpoint
$(\alpha, \beta + q M)$ of $\tau$ satisfies $\alpha \ge \beta + q M$. Hence by \eqref{homo-degree},
\[d_{\tau} \ \ge \ \frac{ q (\beta + q M)  + q m M}{q+m} \ \ge \ 
 \frac{ (\beta + q M)  +  m M}{q+m} \ \ge \ M\]
since $q\ge 1$. If the left endpoint $(\alpha, \beta + q M)$ does not lie on the bisectrix, then we have the
strict inequality $\alpha > \beta + q M$ which in turn implies that the strict inequality $d_{\tau} > M$ holds. 

To prove (b), suppose that $n_{j_{*}} > d(\phi)$ and that there exists a $1\le j \le N$ with $j \not= j_{*}$
such that $d(\phi) \le n_j$. Then $M > 2 d(\phi)$ and so by \eqref{homo-degree},
\[d(\phi) \ \ge \ \frac{q m M}{m+q} \ > \ \frac{ 2 qm}{m+q} d(\phi) \ \ge \ d(\phi)\]
which is a contradiction.
Hence $n_j < d(\phi)$ for all $j\not= j_{*}$. To show that $\xi_{j_{*}} \in {\mathbb Q}$,
we argue by contradiction once again and suppose that the degree of $\xi_{j_{*}}$ over the rationals is at least 2.
Since the conjugates of $\xi_{j_{*}}$ all lie among the roots $\{\xi_j \}_{j=1}^N$, we would be able to find
a conjugate $\xi_j$ with $j\not= j_{*}$. As all conjugates must have the same multiplicity, we see that $n_j = n_{j_{*}}$
which we have just seen is impossible. 

Finally to prove (c),  suppose that $n_{j_{*}} = d(\phi)$ and that there is a $1\le j \le N$ with $j\not= j_{*}$
and $n_j = n_{j_{*}}$. Hence $M \ge 2 n_{j_{*}} = 2 d(\phi)$ and so
\[d(\phi) \ \ge \ \frac{q m M}{m+q} \ \ge \ \frac{ 2 qm}{m+q} d(\phi),\]
implying $2qm/(m+q) \le 1$ and hence $q = m =1$. Plugging this back into \eqref{homo-degree}, we have
\[2d(\phi) \ = \alpha + \beta + M \]
and since $M\ge 2 d(\phi)$, this gives a contradiction unless
$\alpha = \beta = 0$ and $M = 2 d(\phi)$. Hence $d(\phi) = n_{j_{*}} = n_j$ and
\[ \phi_{\rm pr}(x,y) \ = \  c (y - \xi_1 x)^{n} (y - \xi_2 x)^{n} \ = \ c (y^2 + b xy + d x^2)^{n} \]
is  a power of a quadratic form with $n = n_j$. Otherwise we have $n_j < n_{j_{*}} = d(\phi)$ for
all $j\not= j_{*}$ and reasoning as part (b), we also conclude that $\xi_{j_{*}} \in {\mathbb Q}$.

\end{proof}

We will use the notation $m_{\rm pr}(\phi)$ to denote the maximal multiplicity among the roots
$\{\xi_j\}_{j=1}^N$ appearing the factorisation \eqref{phi-factored} of the principal part $\phi_{\rm pr}$
of $\phi$ when $\pi(\phi)$ is a compact edge.
The following theorem contains our basic estimate for $S_0$ which will imply Theorem \ref{main} via Theorem \ref{CoV}. 

\begin{theorem}\label{initial-bound} Suppose that the coefficients of $\phi \in {\mathbb Q}[X,Y]$
are units in ${\mathbb Z}_p$. 

(a) \ If $\pi(\phi)$ is a compact edge, then
\begin{equation}\label{edge-bound}
|S_0(\phi, p^s)| \ \lesssim_{{\rm deg} \, \phi} \ \begin{cases} p^{-s/d(\phi)} & \text{ if } m_{\rm pr}(\phi) < d(\phi) \\
					    s p^{-s/d(\phi)} & \text{ if } m_{\rm pr}(\phi) = d(\phi) \\
					    p^{-s/m_{\rm pr}} & \text{ if } m_{\rm pr}(\phi) > d(\phi) \end{cases}.
\end{equation}

(b) \ If $\pi(\phi)$ is a vertex, then
\begin{equation}\label{vertex}
|S_0(\phi, p^s)| \ \lesssim_{{\rm deg} \, \phi} \ s p^{-s/d(\phi)}
\end{equation}
and this improves to $|S_0(\phi, p^s)| \lesssim p^{-s/d(\phi)}$ when the vertex $\pi(\phi) = (1,1)$. 

(c) \ If $\pi(\phi)$ is an unbounded edge, then
\begin{equation}\label{unbounded}
|S_0(\phi, p^s)| \ \lesssim_{{\rm deg} \, \phi} \ p^{-s/d(\phi)}.
\end{equation}

\end{theorem} 

To see how Theorem \ref{initial-bound} implies Theorem \ref{main} under the assumption $h(\phi) = h(\phi_{\rm pr})$
(which we can make by Theorem \ref{CoV}), we need the following characterisation of $h(\phi_{\rm pr})$.

\begin{proposition}\label{h-homo} Suppose that the principal face $\pi(\phi)$ of $\phi \in {\mathbb Q}[X,Y]$ 
is a compact edge. Then
\[ h(\phi_{\rm pr}) \ = \ \max(d(\phi), m_{{\rm pr}}(\phi)).\]
\end{proposition}

In \cite{W-igusa} it was shown that when $\psi(x,y) = c \, x^{\alpha} y^{\beta} \prod_{j=1}^N (y^q - \xi_j x^m)^{n_j}$
is a quasi-homogeneous polynomial with rational coefficients, then $h(\psi) = \max(d_h(\psi), m_{\mathbb Q}(\psi))$
where
\[ m_{\mathbb Q}(\psi) \ := \ \max(\alpha, \beta, n_j : \xi_j \in {\mathbb Q})\]
and $d_h(\psi)$ is the homogeneous distance given in \eqref{homo-degree}.
This result in \cite{W-igusa} is a minor adjustment of the corresponding euclidean result found in \cite{IM-adapted}.

Note that if the principal face $\pi(\phi)$ 
is a compact edge, then the left endpoint $(\alpha, \beta + q M)$ lies above the bisectrix (so that $\alpha < \beta + qM$)
and the right endpoint $(\alpha + m M, \beta)$ lies below the bisectrix (so that $\beta < \alpha + m M$). Hence
by \eqref{homo-degree} we see that $\max(\alpha, \beta) < d(\phi) $
and so by Lemma \ref{simple-lemma} part (b), we see that 
\begin{equation}
h(\phi_{\rm pr}) \ = \ \max(d_h(\phi_{\rm pr}), m_{\mathbb Q}(\phi_{\rm pr})) \ = \ \max(d(\phi), m_{\rm pr}(\phi)).
\end{equation}
 
Recall that the Varchenko exponent $\nu(\phi) = 1$ if and only $h(\phi)\ge 2$ and
there is an adapted coordinate system in which the principal face is a vertex.
According to Ikromov and M\"uller in \cite{IM-adapted} (Corollaries 2.3 and 4.3), a coordinate system $z = (x,y)$ is adapted
to $\phi(x,y)$
if and only if one of the following conditions  is satisfied:

(a) $\pi(\phi)$ is a compact edge and $m_{\rm pr}(\phi) \le d(\phi)$;

(b) $\pi(\phi)$ is a vertex; or

(c) $\pi(\phi)$ is an unbounded edge.

Hence if the principal face $\pi(\phi)$  of our polynomial $\phi(x,y)$ is a vertex $(d,d)$ (where necessarily
$d = d(\phi)$), then the
coordinates $z = (x,y)$ are adapted and so $\nu(\phi) = 1$ when $d = h(\phi) \ge 2$ and $\nu(\phi) = 0$
when $d=1$. In this case the bound in \eqref{vertex} implies $|S_0(\phi, p^s)| \le C s^{\nu(\phi)} p^{-s/h(\phi)}$,
establishing Theorem \ref{main} in this case. When $\pi(\phi)$ is an unbounded edge, the coordinates are
adapted and hence $d(\phi) = h(\phi)$. Thus \eqref{unbounded} establishes Theorem \ref{main}. 

Next suppose that $\pi(\phi)$ is a compact edge and $m_{\rm pr}(\phi) \not= d(\phi)$. Then by Proposition \ref{h-homo},
the bound \eqref{edge-bound} implies $|S_0(\phi, p^s)| \le C p^{-s/h(\phi)}$ since we are assuming 
$h(\phi) = h(\phi_{\rm pr})$. This establishes Theorem \ref{main} in this case.
Finally suppose that $\pi(\phi)$ is a compact edge and $m_{\rm pr}(\phi) = d(\phi)$. Then $h(\phi) = d(\phi)$ by
Proposition \ref{h-homo}. 
If $\phi \in {\mathcal E}$, then \eqref{edge-bound} establishes
Theorem \ref{main} in this case. 

Hence we may suppose that $\phi \notin {\mathcal E}$. 
In this case, 
we claim that there is a coordinate system in which
the principal face is a vertex so that \eqref{vertex} can be used to show that Theorem \ref{main} holds in
this case as well. Lemma \ref{simple-lemma} part (c) implies that 
\begin{equation}\label{nj<nj*}
\phi_{\rm pr}(x,y) \ = \  c \, x^{\alpha} y^{\beta} \prod_{j=1}^N (y - \xi_j x^m)^{n_j}
\end{equation}
with $\xi_{j_{*}} \in {\mathbb Q}$. Recall that when $\phi_{\rm pr}(x,y) = c (y^2 + b xy + d x^2)^n$
is a power of a quadratic, then
$\phi_{\rm pr}(x,y) =  \ c \, (y - \xi_1 x)^n (y - \xi_2)^n$
where $\xi_1, \xi_2 \in {\mathbb Q}$ since $\phi \notin {\mathcal E}$. This is of the form \eqref{nj<nj*}.
It is simple matter to see that the change of variables $(x,y) \to (x, y + \xi_{j_{*}}^m)$ transforms our
polynomial to one whose principal part is a vertex. 

Therefore Theorem \ref{main} follows from Theorems \ref{CoV} and \ref{initial-bound}.

\section{Proof of Theorem \ref{initial-bound}}\label{proof-basic}

A key step in the proof of the bounds for the local sum $S_0(\phi, p^s)$ in Theorem \ref{initial-bound}
 will be to freeze one of the variables and 
estimate a sum in the other variable.  Equivalently, we will reduce to bounding a one dimensional oscillatory integral and for this,
we will employ the following useful  bound.

\begin{proposition}\label{vc} Let $\psi \in {\mathbb Z}_p [X]$ and suppose there is an $n\ge 1$
such that $\psi^{(n)}(x)/n! \not\equiv 0$ mod $p$ for all $x$ lying in some subset
$S \subseteq {\mathbb Z}/p {\mathbb Z}$. Then there exists a constant $C$, depending on the degree of $\psi$ 
(and not on $S$ or $p$) such that
\begin{equation}\label{vc-bound} 
\Bigl| \sum_{x_0 \in S} \int_{B_{p^{-1}}(x_0)} {\rm e}(p^{-s} \psi(x)) \, dx \Bigr|  \ \le \ C p^{-s/n} 
\end{equation}
holds for all $s\ge 2$. Furthermore when $n =1$,  the sum in \eqref{vc-bound} vanishes.
\end{proposition}

When $S = {\mathbb Z}/p {\mathbb Z}$, Proposition \ref{vc} was proved in \cite{W-jga} using a higher order Hensel lemma.
However the proof given in \cite{W-jga} also gives the strengthening stated here where we consider a truncated
integral (or sum) on which we know some derivative of $\psi$ is non-degenerate. We will outline the proof of
Proposition \ref{vc} in Section \ref{hensel}.

Now let us recall the decomposition \eqref{S_0-decomp} of the oscillatory integral representation of $S_0(\phi, p^s)$
and write $S_0(\phi, p^s) = \sum_{\tau \in \Delta(\phi)} I_{\tau}$ where
\[ I_{\tau} \ := \ \sum_{\substack{\vec l \in \mathbb{N}^2 \\ F(\vec l) = \tau}} p^{-l_1 - l_2} 
\iint_{|x|, |y| = 1} {\rm e}(p^{-s + N(\vec l)} (\phi_{\tau} (x,y) + p g_{\tau}(x,y))) \, dx dy .\]
We will provide a bound for each $I_{\tau}$ with $\tau \in \Delta(\phi)$. We split into two cases:
the case in which $\tau$ is a compact edge and the case in which $\tau$ is a vertex.

\subsection*{When $\tau$ is a compact edge} In this case, $\phi_{\tau}$ is a quasi-homogeneous polynomial
which can be factored
\[\phi_{\tau}(x,y) \ = \   c \, x^{\alpha} y^{\beta} \prod_{j=1}^N (y^q - \xi_j x^m)^{n_j};\]
see \eqref{phi-factored}. If ${\vec l}$ is such that $F({\vec l}) = \tau$, then the vector ${\vec l}$ is perpendicular
to the line $\{(t_1, t_2) :  q t_1 + m t_2 = n \}$ containing $\tau$ if and only if ${\vec l} = l (q, m)$ for some
integer $l \ge 1$. Hence $N(\vec l) = l n$ and so
\begin{equation}\label{Itau} 
I_{\tau} \ = \ \sum_{l \geq 1} p^{-l(m + q)} \iint_{|x|, |y| = 1} {\rm e}(p^{-s + ln} (\phi_{\tau}(x,y) + p g_{\tau}(x,y))) \, dx dy.
\end{equation}
Set $\kappa = \left\lceil \frac{s}{n} \right\rceil - 1$ so that $s = \kappa n + r$ where $1\le r \le n$ and split 
$I_{\tau} = I_{\tau}^1 + I_{\tau}^2$ into two parts where $I_{\tau}^1 = \sum_{l \ge \kappa +1} I_{\tau, l}$
and  $I_{\tau}^2 = \sum_{l \le \kappa} I_{\tau, l}$; here $I_{\tau, l}$ denotes the integral in \eqref{Itau}. Note that
$l\ge \kappa + 1$ precisely when $s - l n \le 0$ and hence the integrand in $I_{\tau, l}$ is identically equal to 1.
Thus $I_{\tau, l} = (1-p^{-1})^2$ for such $l$ and so
\[ I_{\tau}^1 \ = \ (1-p^{-1})^2 \sum_{l \ge \kappa + 1} p^{-l(m+q)} \ \lesssim \  p^{- (\kappa+ 1)(m+q)} \ \le \ p^{- s(m+q)/n}
\ = \ p^{-s/d_{\tau}} .\]
Since $d_{\tau} \le d(\phi)$, we have
\begin{equation}\label{I1}
|I_{\tau}^1|  \ \lesssim \ p^{-s/d(\phi)} 
\end{equation}
which is smaller than the bounds \eqref{edge-bound}, \eqref{vertex}, \eqref{unbounded}
in the statement of Theorem \ref{initial-bound}. Hence \eqref{I1} gives an acceptable contribution for Theorem \ref{initial-bound}.

Let us now concentrate on bounding each integral $I_{\tau,l}$ arising in $I_{\tau}^2$ when $L := s - nl \ge 2$.
In this case we will use Proposition \ref{vc} to bound
\[{\mathcal I}_{\tau, L} \ := \  \iint_{|x|, |y| = 1} {\rm e}(p^{-L} (\phi_{\tau}(x,y) + p g_{\tau}(x,y))) \, dx dy.\]

Set
\[X \ = \ \{(x_0,y_0) \in \left[ \mathbb{Z}/p \mathbb{Z} \right]^2 : x_0 y_0  \neq 0\}\] 
and note that the region of integration in the integral ${\mathcal I}_{\tau, L}$ is precisely the set of 
$(x,y) \in {\mathbb Z}_p^2 $ such that $(x,y)$ is congruent mod $p$ to an 
element of $X$. 

We will split $X = Z_0 \cup Z_1 \cup \cdots \cup Z_N$ according to roots $\{\xi_j\}$ of $\phi_{\tau}$. All the roots arising
from the quasi-homogeneous polynomials $\phi_{\tau}$ with $\tau \in \Delta(\phi)$ are algebraic numbers and hence lie in a finite field
extension of ${\mathbb Q}_p$. Therefore each $p$-adic absolute value $|\cdot| = |\cdot|_p$ extends uniquely to these elements.
Elementary considerations (see \cite{W-igusa}) allow us to identify a finite collection of primes so that for all
primes $p$ not in this collection, $|\xi_j|_p = 1$ and $|\xi_j - \xi_k|_p = 1$ whenever $j\not= k$. We will include this
finite set into our exceptional set of primes ${\mathcal P}$.

We define
\begin{IEEEeqnarray*}{rCll}
Z_0 & := & \{(x_0,y_0) \in X : \phi_{\tau}(x_0,y_0) \not\equiv 0 \ {\rm mod} \ p  \} & \ \ \text{and} \\
Z_j & := & \{(x_0,y_0) \in X : |y_0^q - \xi_j x_0^m| < 1\}  & \text{ for $1 \leq j \leq N$} .
\end{IEEEeqnarray*}
Note that $Z_j$ may be empty if there are no ordered pairs of elements of $(x_0, y_0) \in (\mathbb{Z}/p \mathbb{Z})^2$ for which the inequality defining $Z_j$ holds. Furthermore, the $Z_j$ are disjoint: if $|y_0^q - \xi_j x_0^m| < 1$ and $j \neq j'$ then the ultrametric inequality shows that $|y_0^q - \xi_{j'} x_0^m| = |y_0^q - \xi_j x_0^m + \xi_j x_0^m - \xi_{j'} x_0^m| = 1$, since $|\xi_j x_0^m - \xi_{j'} x_0^m| = 1$ by the separation of the roots.


This gives us a disjoint decomposition of $X$. Accordingly, we split ${\mathcal I}_{\tau, L} = \sum_{j=0}^N {\mathcal I}_j$ where
\[{\mathcal I}_{j} \ := \  \sum_{(x_0, y_0) \in Z_j} \iint_{B_{p^{-1}}(x_0, y_0)} 
{\rm e}(p^{-L} (\phi_{\tau}(x,y) + p g_{\tau}(x,y))) \, dx dy .\]
Here $B_{p^{-1}}(x_0, y_0) = \{ (x,y) \in {\mathbb Z}_p^2 : \max(|x-x_0|, |y-y_0|) \le p^{-1}\}$ consists of those
elements of $\mathbb{Z}_p^2$ that are congruent to $(x_0, y_0)$ modulo $p$.

First we claim that ${\mathcal I}_0 = 0$. In fact each integral appearing in the sum defining ${\mathcal I}_0$ vanishes.
Fix $(x_0, y_0) \in Z_0$ and let $I_{x_0,y_0}$ denote the corresponding integral in ${\mathcal I}_0$. By
a simple extension of  Euler's homogeneous function theorem to the 
quasi-homogeneous case, we have $(qx, m y) \cdot \nabla \phi_{\tau} (x,y) = n \phi_{\tau}(x,y)$ and so
$\nabla \phi_{\tau} (x_0, y_0) \not\equiv 0$ mod $p$. Set
$\varphi(x,y) = \phi_{\tau}(x,y) + p g_{\tau}(x,y)$ and note that $\nabla \varphi(x_0, y_0) \not\equiv 0$ mod $p$.
We write $I_{x_0, y_0} =$
\begin{IEEEeqnarray*}{CCl}
 & \sum_{\substack{(u_0, v_0) \in (\mathbb{Z}/p^{L-1} \mathbb{Z})^2 \\ (u_0, v_0) \equiv (x_0, y_0) \text{ mod $p$}}} & \iint_{B_{p^{-L + 1}}(u_0, v_0)} {\rm e} (p^{-L}\varphi(x,y)) \, dx dy \\
= & \sum_{\substack{(u_0, v_0) \in (\mathbb{Z}/p^{L-1} \mathbb{Z})^2 \\ (u_0, v_0) \equiv (x_0, y_0) \text{ mod $p$}}} (p^{-(L+1)})^ 2 & \iint_{\mathbb{Z}_p^2} {\rm e}(p^{-L} \varphi(u_0 + p^{L-1} u, v_0 + p^{L-1} v)) \, du dv
\end{IEEEeqnarray*}
and note that 
\[  \varphi(u_0 + p^{L-1} u, v_0 + p^{L-1} v)) \equiv \varphi(u_0, v_0) + p^{L-1} \nabla\varphi(u_0, v_0) \cdot (u,v) \ {\rm mod} \ p^L \]
since $L\ge 2$. Hence $I_{x_0, y_0} = $
\[\sum_{\substack{(u_0, v_0) \in (\mathbb{Z}/p^{L-1} \mathbb{Z})^2 \\ (u_0, v_0) \equiv (x_0, y_0) \text{ mod $p$}}} (p^{-L + 1})^2 
{\rm e}(p^{-L} \varphi(u_0, v_0)) \iint_{\mathbb{Z}_p^2} \chi(p^{-1} \nabla \varphi(u_0, v_0) \cdot (u,v)) \, du dv\]
and each integral above vanishes because $\nabla \varphi(u_0, v_0) \not\equiv 0$ mod $p$.  Thus ${\mathcal I}_0 = 0$.

Let us now examine the other terms ${\mathcal I}_j, \,  1\le j \le N$. We have
\begin{IEEEeqnarray*}{rCCl}
{\mathcal I}_j \ & = & \sum_{(x_0, y_0) \in Z_j} & \iint_{B_{p^{-1}}(x_0, y_0)} {\rm e}(p^{-L}\varphi(x,y)) \, dx dy \\
& = & \sum_{x_0 \in \mathbb{Z}/p \mathbb{Z} \setminus \{ 0 \}} \sum_{y_0 \in Z_{j, x_0}} & \iint_{B_{p^{-1}}(x_0, y_0)} {\rm e}(p^{-L} \varphi(x,y)) \, dx dy
\end{IEEEeqnarray*}
where $Z_{j, x_0} = \{y_0 \in \mathbb{Z}/p\mathbb{Z} \setminus \{ 0 \} : (x_0, y_0) \in Z_j \}$. 

Interchanging the sum in $y_0$ and the $x$ integration, we have 
\[{\mathcal I}_j  \ = \ \sum_{x_0 \in \mathbb{Z}/p\mathbb{Z} \setminus \{ 0 \}} \int_{B_{p^{-1}}(x_0)} 
\Bigl( \sum_{y_0 \in Z_{j, x_0} } \int_{B_{p^{-1}}(y_0)} {\rm e}(p^{-L} \varphi(x,y)) \, dy \Bigr) \, dx.\]
Denoting Inner$_{x_0}(x)$ as the sum in $y_0$, we have
\[{\mathcal I}_j \ = \ \sum_{x_0 \in \mathbb{Z}/p \mathbb{Z} \setminus \{ 0 \}} \int_{B_{p^{-1}}(x_0)} \text{Inner}_{x_0}(x) \, dx .\]
For any fixed $x_0 \in \mathbb{Z}/p \mathbb{Z} \setminus \{0\}$ and $x$ such that $|x - x_0| \leq p^{-1}$, define $\psi_x(y)$ to be the function 
$\varphi(x,y)$. Thus we have
\[\text{Inner}_{x_0}(x) \ = \ \sum_{y_0 \in Z_{j,x_0}} \int_{B_{p^{-1}(y_0)}} {\rm e}(p^{-L} \psi_x(y)) \, dy,\]
putting us in a position to employ our  bound \eqref{vc-bound} since it is straightforward to check that
$\psi_x^{(n_j)} (y_0) / n_j !\not\equiv 0$ mod $p$ for every $y_0 \in Z_{j, x_0}$. Hence the uniform bound
\[|\text{Inner}_{x_0}(x)| \ \leq \ C_{\deg \phi} \ p^{-L/n_j}\]
holds and when $n_j = 1$, we have in fact ${\rm Inner}_{x_0}(x) = 0$. This implies that
\begin{equation}\label{Ij-0}
|{\mathcal I}_j| \ \leq \ C_{\deg \phi}\  p^{-L/n_j} \ \ {\rm but \ when} \ n_j = 1, \ \ {\mathcal I}_j = 0 .
\end{equation}
Therefore $|{\mathcal I}_{\tau, L}| \leq C_{\deg \phi} \,  p^{-L/m_{\tau}}$, where $m_{\tau}$ 
is the maximal multiplicity of the roots $\{\xi_j\}$ of $\phi_{\tau}$. This gives us a bound on the sum of
those terms in $I_{\tau}^2$ where $L = s - l n \ge 2$; write $I_{\tau}^2 = I_{\tau}^{2,1} + I_{\tau}^{2,2}$
where 
$I_{\tau}^{2,1} \ := \  \sum_{\substack{1\le l \le \kappa \\ s - l n \ge 2}} I_{\tau, l} $
so that
\[
| I_{\tau}^{2,1}| \ \le \ C \, 
\Bigl(\sum_{1 \le l \le \kappa} p^{-l (m + q)} p^{ln/m_{\tau}} \Bigr) p^{-s/m_{\tau}} \ = \
\Bigl(\sum_{1 \le l \le \kappa} p^{-l n [\frac{1}{d_{\tau}} - \frac{1}{m_{\tau}}]} \Bigr) p^{-s/m_{\tau}} .
\]
Hence
\begin{equation}\label{I2-1}
|I_{\tau}^{2,1}| \ \lesssim_{\deg \phi} \ \begin{cases} p^{-s/d_{\tau}} & \text{ if } m_{\tau} < d_{\tau} \\
					    s p^{-s/d_{\tau}} & \text{ if } m_{\tau} = d_{\tau} \\
					    p^{-s/m_{\tau}} & \text{ if } m_{\tau} > d_{\tau} 
\end{cases}.
\end{equation}

If $\tau = \pi(\phi)$ is the principal face, then $d_{\tau} = d(\phi)$ and $m_{\tau} = m_{\rm pr}(\phi)$ so that
\eqref{I2-1} gives an acceptable contribution to the bound \eqref{edge-bound} in Theorem \ref{initial-bound}. Now suppose
that the compact edge $\tau$ is not the principal face. By Lemma \ref{simple-lemma} part (a), we conclude
that $m_{\tau} \le d_{\tau}$. Furthermore,
if the endpoint of $\tau$ does not lie on the bisectrix, then in fact $m_{\tau} < d_{\tau}$ and so \eqref{I2-1} implies 
$|I_{\tau}^{2,1}| \lesssim p^{-s/d_{\tau}} \lesssim p^{-s/d(\phi)}$ and this is an acceptable bound as before.

Finally suppose that an endpoint of $\tau$ lies on the bisectrix. Then the principal face $\pi(\phi)$ is a vertex and $d_{\tau} = d(\phi)$.
The bound \eqref{I2-1} gives an acceptable contribution to the bound \eqref{vertex} in Theorem \ref{initial-bound}
{\bf unless} the vertex $\pi(\phi)$ is $(1,1)$. In this case $m_{\tau} = d_{\tau} = d(\phi) = 1$ and for formula \eqref{homo-degree}
for $d_{\tau}$ shows two possible outcomes: (1) either $q = M = \alpha = 1$ and $\beta = 0$ in which case $\phi_{\tau} = a x (y - \xi x^m)$ for some
$\xi \in {\mathbb Q}$ or (2)  $q = M = m = \beta = 1$ and $\alpha = 0$ in which case $\phi_{\tau} = a y (y - \xi x)$ for some $\xi \in {\mathbb Q}$.
In either case, ${\mathcal I}_{\tau, L} = {\mathcal I}_0 + {\mathcal I}_j$ where $n_j = 1$. Hence by \eqref{Ij-0} we see
that ${\mathcal I}_{\tau, L} = 0$ implying in turn $I_{\tau}^{2,1} = 0$ in this case.

It remains to treat $I_{\tau}^{2,2}$ where we are summing the integrals $I_{\tau, l}$ for $1\le l \le \kappa$
and $s - ln = 1$. The condition $s - l n = 1$ can only occur if
$l = \kappa$ and $s \equiv 1$ mod $n$. Hence $I_{\tau}^{2,2} = I_{\tau, \kappa}$ and $s - \kappa n = 1$ so that
\[I_{\tau}^{2,2} \ = \ p^{-\kappa(m + q)} \iint_{|x|, |y| = 1} {\rm e}(p^{-1} (\phi_{\tau}(x,y))) \, dx dy \]
which is an exponential sum over a finite field. We claim that the bound
\begin{equation}\label{field-acceptable}
|I_{\tau}^{2,2}| \ \lesssim_{\phi} \ p^{-s/d(\phi)}
\end{equation}
holds and as we have seen before,  this is an acceptable bound.

First we can apply the Weil bound \cite{Weil} for finite field sums (say to the $y$ integral)
to see that
\[
| I_{\tau}^{2,2}| \ \lesssim_{\phi} \ \,  p^{-(\frac{1}{2} - \frac{1}{d_{\tau}})} p^{-s/d_{\tau}};
\]
here we used the identity $\kappa (m + q) = (s-1) (m + q)/n = (s-1)/d_{\tau}$. Therefore if $d_{\tau} \ge 2$, we obtain
the bound \eqref{field-acceptable}. We now treat the case when $d_{\tau} < 2$. 

First suppose that $\tau$ is not the principal face. Then $m_{\tau} \le d_{\tau}$
by Lemma \ref{simple-lemma} which implies $m_{\tau} =1$ (and hence $d_{\tau} \ge 1$)
so that $\phi_{\tau} (x,y) = a x^{\alpha} y^{\beta} (y^q - \xi x^m)$
for some $\xi \in {\mathbb Q}$. If $\tau$  lies below
the bisectrix, then the second coordinate of the left endpoint must be equal to 1. Hence $\beta  + q = 1$ implying $q =1$, $\beta =0$
and so $\phi_{\tau}(x,y) = a x^{\alpha} (y - \xi x^m)$. Similarly if $\tau$ lies above the bisectrix, then $\alpha=0$, $m=1$ and
so $\phi_{\tau}(x,y) = a y^{\beta} (y - \xi x)$. In either case $\phi_{\tau}$ is either linear in $y$ or linear in $x$ which implies
that one of the integrals
\[ \int_{|y| \le  1} {\rm e}(p^{-1} (\phi_{\tau}(x,y))) \,  dy, \ \ \ \  \int_{|x|\le 1} {\rm e}(p^{-1} (\phi_{\tau}(x,y))) \, dx \]
is equal to 0. Hence 
\[ \Bigl| \iint_{|x|, |y| = 1} {\rm e}(p^{-1} (\phi_{\tau}(x,y))) \, dx dy \Bigr| \ \lesssim \ p^{-1} \]
which in turn implies
\[
| I_{\tau}^{2,2}| \ \le \ C_{\phi} \  p^{-(1 - \frac{1}{d_{\tau}})} p^{-s/d_{\tau}} \ \le \ C_{\phi} \ p^{-s/d(\phi)},
\]
establishing \eqref{field-acceptable} in this case.

Now suppose that $\tau$ is the principal face. Then $d_{\tau} = d(\phi)$. If $d_{\tau} < 1$, then $\tau$ cannot contain
any lattice points away from the coordinate axes. Hence $\alpha = \beta =0$, $M = 1$ and so
$\phi_{\tau}(x,y) = a (y^q - \xi x^m)$ for some $\xi \in {\mathbb Q}$. Using the formula \eqref{homo-degree},
$d_{\tau} = qm/(m+q)$ and the restriction $d_{\tau}  < 1$ shows $q=1$. Therefore
 \begin{equation}\label{m-bound} \iint_{|x|, |y| = 1} {\rm e}(p^{-1} (\phi_{\tau}(x,y))) \, dx dy  \ = \  - p^{-1} 
\int_{|x| = 1} {\rm e}(p^{-1} a x^m) \, dx    
\end{equation}
and if $m=1$, the above integral is $O(p^{-2})$ leading to the bound 
\[
| I_{\tau}^{2,2}| \ \le \ C_{\phi} \  p^{-(2- \frac{1}{d_{\tau}})} p^{-s/d_{\tau}} \ \le \ C_{\phi} \ p^{-s/d(\phi)}
\]
which proves \eqref{field-acceptable}. When $m\ge 2$,  we are stuck with the bound $O(p^{-3/2})$ arising from a character sum estimate for the integral in
\eqref{m-bound} but in this case, we have $d_{\tau} = m/(m+1) \ge 2/3$ and so 
\[
| I_{\tau}^{2,2}| \ \le \ C_{\phi} \  p^{-(\frac{3}{2}- \frac{1}{d_{\tau}})} p^{-s/d_{\tau}} \ \le \ C_{\phi} \ p^{-s/d(\phi)}
\]
which once again proves \eqref{field-acceptable}. 

Finally suppose that $\tau$ is the principal face but $d_{\tau} \ge 1$. Since 
\[ \iint_{|x|, |y| = 1} {\rm e}(p^{-1} (\phi_{\tau}(x,y))) \, dx dy\ = \
\iint_{{\mathbb Z}_p^2} {\rm e}(p^{-1} (\phi_{\tau}(x,y))) \, dx dy \ + \ O(p^{-1}) , \]
we can use Cluckers's bound in \cite{Cluckers} to conclude that 
\[ \Bigl| \iint_{|x|, |y| = 1} {\rm e}(p^{-1} (\phi_{\tau}(x,y))) \, dx dy \Bigr| \ \lesssim \ p^{-1/d_{\tau}} \]
which implies 
\[
| I_{\tau}^{2,2}| \ \le \ C_{\phi} \  p^{-s/d_{\tau}} \ \le \ C_{\phi} \ p^{-s/d(\phi)},
\]
establishing \eqref{field-acceptable} in all cases. 

\subsection*{When $\tau$ is a vertex} 
We will now consider the case where $\tau  = (\alpha, \beta)$ is a vertex. This means that $\phi_{\tau}(x,y) = c x^{\alpha} y^{\beta}$ 
a monomial. However, the sum over $\vec l$ will consist of more than just integer multiples of a fixed vector. 

Assume that $\tau$ is the endpoint of two edges $e_1$ and $e_2$, where $e_2$ lies below (to the right of) $\tau$ and
$e_1$ lies above (to the left of) $\tau$. Hence if the edges are compact,
\[ e_1 \subset \{ (t_1, t_2) : q_1 t_1 + m_1 t_2 = n_1 \} \ \ {\rm and} \ \ 
e_2 \subset \{ (t_1, t_2) : q_2 t_1 + m_2 t_2 = n_2 \} \]
for some positive integers $(q_j, m_j), j=1,2$ with ${\rm gcd}(q_j, m_j) = 1$. If the $e_2$ is unbounded
(that is, it is a horizontal line), then $e_2 \subset \{(t_1, t_2) : t_2 = \beta \}$. Likewise if $e_1$ is unbounded (vertical),
then $e_1 \subset \{(t_1, t_2) : t_1 = \alpha \}$.

If both edges $e_1$ and $e_2$ are compact, then $F(\vec l) = \tau$ if and only if $\vec l = (l_1, l_2)$ satisfies 
\[\frac{m_1}{q_1} < \frac{l_2}{l_1} < \frac{m_2}{q_2} .\]
If one of the edges is unbounded, the corresponding upper or lower restriction of the ratio $l_2/l_1$ is removed; for example,
if $e_2$ is an infinite horizontal edge and $e_1$ is compact, then $F(\vec l) = \tau$ if and only if $m_1/q_1 < l_2 / l_1. $
We will, without loss of generality, assume that $\alpha \leq \beta$. 

Then $N(\vec l) = l_1 \alpha + l_2 \beta$ and our integral $I_{\tau}$ to bound is
\begin{IEEEeqnarray*}{rCl}
I_{\tau} & = & \sum_{\substack{\vec l \in \mathbb{N}^2 \\F(\vec l) = \tau}} p^{-l_1 - l_2} \iint_{|x|, |y| = 1} {\rm e}(p^{-s + N(\vec l)} (c x^{\alpha} y^{\beta} + p g_{\tau}(x,y))) \, dx dy \\
& = &  \sum_{\substack{l_1, l_2 \geq 1 \\ l_1 \frac{r_1}{t_1} < l_2 < l_1 \frac{r_2}{q_2}}} p^{-l_1 - l_2} \iint_{|x|, |y| = 1} 
{\rm e}(p^{-s + N(\vec l)} (c x^{\alpha} y^{\beta} + p g_{\tau}(x,y))) \, dx dy 
\end{IEEEeqnarray*}
with the understanding that if one of edges $e_1$ and/or $e_2$ is unbounded, then the corresponding restriction  on the ratio
$l_2/l_1$ does not appear.

We decompose $I_{\tau} = I_{\tau}^1 + I_{\tau}^2 + I_{\tau}^3$  into three pieces according
to whether $N(\vec l) \ge s$, $N(\vec l) = s-1$ and $N(\vec l) \le s-2$, respectively. When $N(\vec l) \ge s$, the integrand is
identically equal to 1 and so
\[I_{\tau}^{1}  \ \le \  (1 - p^{-1})^2 \sum_{\substack{l_1, l_2 \geq 1 \\ N(\vec l) \geq s \\  l_2/l_1 < r_2/q_2}} p^{-l_1 - l_2} \]
if the edge $e_2$ is compact. When $e_2$ is unbounded, the only restriction on the sum over $\vec l = (l_1, l_2)$ with $l_1, l_2 \ge 1$ is
 $N(\vec l) = l_1 \alpha + l_2 \beta \ge s$. 

Suppose $e_2$ is compact so that $d_{e_2} = n_2/(m_2 + q_2)$ and 
$\alpha q_2 + \beta m_2 = n_2$ since $\tau = (\alpha, \beta) \in e_2$. First consider the case $\alpha < \beta$. Then since
$l_2 = - \alpha l_1 /\beta + N/\beta$ where $N = N(\vec l)$, we have
\[ l_2/l_1 \ = \ - \alpha/\beta + N/\beta l_1  \  < \ m_2/q_2 \ \ \ {\rm so} \ \ \ \frac{N q_2}{q_2 \alpha + m_2 \beta} \ = \ \frac{N q_2}{n_2} \ < \ l_1.\]
Hence in this case, since $l_1 + l_2 = N/\beta + (1 - \alpha/\beta) l_1$,
\begin{equation}\label{a<b}
|I_{\tau}^1| \ \lesssim \ \sum_{N \ge s} p^{-N/\beta}  \sum_{l_1 > N q_2/ n_2} p^{-(1 - \alpha/\beta) l_1} \ \lesssim \
\sum_{N\ge s} p^{-N/d_{e_2}} \ \lesssim \ p^{-s/d_{e_2}}.
\end{equation}
When $\alpha = \beta$, then $\beta = d_{e_2} = d(\phi)$ and $N/d_{e_2} = l_1 + l_2$ so that 
\begin{equation}\label{a=b}
|I_{\tau}^1| \ \lesssim \ \sum_{N\ge s } p^{-N/d_{e_2}} \sum_{{\vec l} \, : \,  l_1 + l_2 = N/d_{e_2}} \, 1  \ \lesssim \ 
\sum_{N\ge s} N p^{-N/d_{e_2}} \ \lesssim \ s \, p^{-s/d_{e_2}} .
\end{equation}

Suppose now that $e_2$ is an unbounded (horizontal) edge so that $\beta = d_{e_2} = d(\phi)$. We again have
$l_1 \alpha + l_2 \beta = N(\vec l) = N$ so that
\[ |I_{\tau}^1| \ \lesssim \ \sum_{N \ge s} p^{-N/\beta}  \sum_{l_1 \ge 1} p^{-(1 - \alpha/\beta) l_1} , \]
implying that  \eqref{a<b} holds if $\alpha < \beta$ and \eqref{a=b} holds if $\alpha = \beta$. Hence it all cases, we have
\begin{equation}\label{a-b}
|I_{\tau}^1| \ \lesssim \ \begin{cases} p^{-s/d_{e_2}} & \text{ if } \alpha \not= \beta\\
					    s p^{-s/d_{e_2}} & \text{ if } \alpha = \beta
					    
\end{cases} .
\end{equation}

Next let us turn our attention to 
\[I_{\tau}^{3}  \ = \ \sum_{\substack{l_1, l_2 \geq 1 \\ s - N(\vec l) \geq 2 \\ m_1/q_1 < l_2/l_1 < m_2/q_2}} p^{-l_1 - l_2} \iint_{|x|, |y| = 1} 
{\rm e}(p^{-s + N (\vec l)} \varphi_{\tau}(x,y) ) \, dx dy \]
where $\varphi_{\tau}(x,y) =  c x^{\alpha} y^{\beta} + p g_{\tau}(x,y)$. Since $\nabla \varphi_{\tau} (x,y) \not\equiv 0$ mod $p$ for any
$(x,y) \in {\mathbb Z}_p^2$ satisfying $|x| = |y| = 1$, the same argument above showing that ${\mathcal I}_0 = 0$
shows that $I_{\tau}^3 = 0$. 

Finally we treat 
\[ I_{\tau}^2  \ = \  \sum_{\substack{l_1, l_2 \geq 1 \\ N(\vec l) = s - 1 \\ r_1/t_1 < l_2/l_1 < r_2/q_2}} p^{-l_1 - l_2} 
\iint_{|x|, |y| = 1} {\rm e}(p^{-1} c x^{\alpha} y^{\beta})  \, dx dy \]
by using the same estimates establishing \eqref{a<b} and \eqref{a=b} although we lose a bit since we need to replace $s$ by $s-1$
in these estimates. On the other hand we gain 
a factor $O(p^{-1/2})$ from the above integral by an easy variant of Weil's bound (a direct computation suffices here;
see \cite{Vino-elementary}). Hence we obtain
\begin{equation}\label{a-b-2}
|I_{\tau}^2| \ \lesssim \  p^{-(\frac{1}{2} - \frac{1}{d_{e_2}})} \begin{cases} p^{-s/d_{e_2}} & \text{ if } \alpha \not= \beta\\
					    s p^{-s/d_{e_2}} & \text{ if } \alpha = \beta
					    
\end{cases} 
\end{equation}
which implies the same bound in \eqref{a-b} when $d_{e_2} \ge 2$. When $d_{e_2} < 2$, then $\alpha = 1$ and so
$1\le d_{e_2} < 2$ (we are assuming
the vertex $\tau  = (\alpha, \beta)$ lies on or above the bisectrix). Therefore the integral in $I_{\tau}^2$ can be computed;
\begin{equation}\label{computation}
\iint_{|x|, |y| = 1} {\rm e}(p^{-1} c x y^{\beta})  \, dx dy =  (1 - p^{-1}) \int_{|z|=1} {\rm e}( c z) \, dz = 
-(1-p^{-1}) p^{-1}. 
\end{equation}
Here we simply made the change of variables $ z = x y^{\beta}$ in the $x$ integral. This improves the bound in \eqref{a-b-2} to
\[
|I_{\tau}^2| \ \lesssim \  p^{-(1 - \frac{1}{d_{e_2}})} \begin{cases} p^{-s/d_{e_2}} & \text{ if } \alpha \not= \beta\\
					    s p^{-s/d_{e_2}} & \text{ if } \alpha = \beta
					    
\end{cases} 
\]
and so we obtain the same bound in \eqref{a-b} in all cases. Hence 
\begin{equation}\label{a-b-2-improved}
|I_{\tau}| \ \lesssim \  \ \begin{cases} p^{-s/d_{e_2}} & \text{ if } \alpha \not= \beta\\
					    s p^{-s/d_{e_2}} & \text{ if } \alpha = \beta
					    
\end{cases} 
\end{equation}
and since $d_{e_2} \le d(\phi)$ and the case $\alpha = \beta$ only occurs if the vertex $\tau$ is the principal face, we see that
\eqref{a-b-2-improved} gives an acceptable contribution to the estimates in Theorem \ref{initial-bound} except in the solitary
case that the principal face of $\phi$ is $(1,1)$ where we need to improve the bound for $I_{\tau}$ to $|I_{\tau}| \lesssim p^{-s}$ in order
to finish the proof of Theorem \ref{initial-bound}.

\subsection*{The last step}
When the vertex $\tau = (1,1)$, then $\tau = \pi(\phi)$ and $d(\phi) =1$. Here we will show the
improved bound $|I_{\tau}| \lesssim p^{-s}$ which will conclude the proof of Theorem \ref{initial-bound}.

Recall the decomposition $I_{\tau} = I_{\tau}^1 + I_{\tau}^2 + I_{\tau}^3$
above where $I_{\tau}^3 = 0$ and in this case,
\begin{equation}\label{(1,1)-1}
I_{\tau}^{1}  \ = \  (1 - p^{-1})^2 \sum_{\substack{{\vec l} : \, N(\vec l) \geq s \\  m_1/q_1 <  l_2/l_1 < m_2/q_2}} p^{-l_1 - l_2} 
\end{equation}
with the understanding that $l_1, l_2\ge 1$ and if one of edges $e_1$ and/or $e_2$ is unbounded, then the corresponding restriction  on the ratio
$l_2/l_1$ does not appear. Furthermore using \eqref{computation}, we have
\begin{equation}\label{(1,1)-2}
I_{\tau}^2  \ = \  -(1-p^{-1}) p^{-1}  \sum_{\substack{{\vec l}: \,  N(\vec l) = s - 1 \\ r_1/t_1 < l_2/l_1 < r_2/q_2}} p^{-l_1 - l_2}  .
\end{equation}
Thus we see that in this case (when $\tau = (1,1)$),  $I_{\tau}$ is a difference of two explicit sums of positive terms. A careful
examination of this difference will exhibit the additional cancellation we seek. 

We will show this when the edges $e_1$ and $_2$ are both infinite so the restrictions on ${\vec l} = (l_1, l_2)$ are 
$l_1, l_2 \ge 1$ and either $N(\vec l) \ge s$ or $N(\vec l) = s-1$. The case when one edge (or both) is compact is similar. In this case,
$N(\vec l) = l_1 + l_2$ and so
\[ I_{\tau}^1 \ = (1 - p^{-1})^2 \sum_{\substack{l_1, l_2 \ge 1 \\  l_1 + l_2 \ge s}} p^{-l_1 - l_2} \ = \
 (1 - p^{-1})^2 \sum_{N \ge s} (N-1) p^{-N} \]
and by the geometric series formula, 
\[
I_{\tau}^1 \ = \ (s-1) p^{-s}  \ + \ - (s-2) p^{-s-1}.
\]
In a similar but easy manner,
\[
 I_{\tau}^2 \ = \  - (1 - p^{-1}) p^{-1} \sum_{\substack{l_1, l_2 \ge 1 \\  l_1 + l_2 = s-1}} p^{-l_1 - l_2} \ = \
 - (1 - p^{-1}) (s-2) p^{-s} 
\]
and so $I_{\tau} = I_{\tau}^1 + I_{\tau}^2 = p^{-s}$ which shows the desired cancellation between the
two terms $I_{\tau}^1$ and $I_{\tau}^2$. 

This completes the proof of Theorem \ref{initial-bound}.

\section{Proof of Theorem \ref{CoV}}\label{sect-CoV}

Here we give the proof of Theorem \ref{CoV} by developing an appropriate variant of
an algorithm due Ikromov and M\"uller in \cite{IM-adapted} which produces an adapted 
coordinate system for any real-analytic function $f$. This algorithm constructs a series of changes
of variables, and except for the final one, all are given by a simple polynomial map. The goal
will be to show that the polynomial change of variables reached by the penultimate stage 
satisfies the conclusion of Theorem \ref{CoV}.

\subsection{Conditions for Adapted Coordinate Systems}
For this section we will work entirely with \emph{real-analytic} functions $f$.
We will observe what happens when we apply the algorithm from \cite{IM-adapted} to a polynomial with rational coefficients.  

The key observation is the one made in \cite{W-igusa}: Corollary 2.3 from \cite{IM-adapted} is valid in any field $K$. 
The content of this corollary is to relate the roots of a quasi-homogeneous polynomial $f$ to its 
{\it homogeneous distance} $d(f)$. 
A polynomial $f \in K[X,Y]$ being quasi-homogeneous makes sense in any field $K$ and can be factored as
\[f(x,y) \ = \ c \, x^{\alpha} y^{\beta} \prod_{j=1}^N (y^q - \xi_j x^m)^{n_j}\]
where $c \in K$ and the  roots $\{\xi_j\}_{j=1}^N$ lie in some finite field extension of $K$. Here
${\rm gcd}(m,q) = 1$ and $\kappa_1 := q/n, \kappa_2 := m/n$ are the dilation parameters so that
$f(r^{\kappa_1} x, r^{\kappa_2} y) = r f(x,y)$ for $r>0$. Recall that the  homogeneous distance of $f$ is defined as
\[d(f) \ = \ \frac{1}{\kappa_1 + \kappa_2} \ = \ \frac{q \alpha + m \beta + q m M}{q + m}\]
where $M := \sum_{j=1}^N n_j$. Finally set $n_0 = \alpha$ and $n_{N+1} = \beta$.

We now reproduce the version Corollary 2.3 from \cite{IM-adapted} as it appeared in \cite{W-igusa}.
\begin{lemma}[\cite{IM-adapted}, \cite{W-igusa}]\label{QuasiHomogeneousHeight}
Let $f \in K[X,Y]$ be a quasi-homogeneous polynomial as above. Without loss of generality
suppose that $\kappa_2 \geq \kappa_1$ or $1\le q \le m$.
\begin{enumerate}
\item If there is a multiplicity $n_{j_*} > d(f)$ for some $0 \leq j_* \leq N+1$, then all the other multiplicities must be strictly less than $d(f)$; that is, $n_j < d(f)$  for all $0 \leq j \neq j_* \leq N+1$. In particular, there is at most one multiplicity $n_j$, 
$0 \leq j \leq N+1$ with $n_j > d(f)$.
\item If $\kappa_2/\kappa_1 \notin \mathbb{N}$, then $M = \sum_{j=1}^N n_j < d(f)$.
\item If $\kappa_2/\kappa_1 \in \mathbb{N}$, then $n_j \leq d(f)$ for any $1 \leq j \leq N$ such that $\xi_j \notin K$. 
\end{enumerate}
\end{lemma}
The corollary says that the multiplicity of every root $\xi_j, 1\le j \le N,$ is bounded by $d(f)$ unless
$\kappa_2/\kappa_1 \in {\mathbb N}$, 
in which case there is at most one root $\xi_j, 1\le j \le N$ with multiplicity exceeding $d(f)$. If
such a root exists, it necessarily lies in $K$ and we shall call it the {\it principal root} of $f$.

We will need the following theorem in \cite{IM-adapted}.

\begin{theorem}[Ikromov-M\"uller]\label{AdaptedCharacterization}
Let $f$ be a real-analytic function near the origin with $f(0,0) = 0$ and $\nabla f(0,0) = 0$.
Then the given coordinates are not adapted to $f$ if and only if the following hold true:
\begin{enumerate}
\item The principal face $\pi(f)$ of the Newton polyhedron is a compact edge. It thus lies on a uniquely
determined line $\kappa_1 t_1 + \kappa_2 t_2 = 1$ with $\kappa_1, \kappa_2 > 0$. Swapping coordinates
if necessary, we may assume $\kappa_2 \ge \kappa_1$.
\item $\frac{\kappa_2}{\kappa_1} \in \mathbb{N}$. Note that this implies that $q = 1$ in \eqref{phi-factored}.
\item The inequality $m_{\text{pr}}(f) > d(f)$ holds.
\end{enumerate}
Moreover, in this case, an adapted coordinate system for $f_{\text{pr}}$ is given by $y_1 := x_1$, $y_2 := x_2 - a x_1^m$, where $a$ is the root of $f_{\text{pr}}$ in the sense of \eqref{phi-factored} with the maximum multiplicity. The height of $f_{\text{pr}}$ is then given by $h(f_{\text{pr}}) = m_{\text pr}(f)$.
\end{theorem}

We will apply Theorem \ref{AdaptedCharacterization} in the case when $f$ has rational coefficients.
In this case, when the principal face is a compact edge, 
$m_{\text{pr}}(f) = \max_{1\le j \le N} n_j$ where the $\{n_j\}$ are the multiplicities
of the roots of the principal part $f_{\text{pr}}(x,y) = c x^{\alpha} y^{\beta} \prod_{j=1}^N (y^q - \xi_j x^m)^{n_j}$
of $f$. We have $f_{\text{pr}} \in {\mathbb Q}[X,Y]$ is a quasi-homogeneous polynomial with rational
coefficients and we apply Lemma \ref{QuasiHomogeneousHeight} with $K = {\mathbb Q}$ to conclude
that if $n_{j_{*}} = m_{\text{pr}}(f) > d(f)$, then the principal root $\xi_{j_{*}} \in {\mathbb Q}$ of $f$ is a rational number.


We will adopt the following terminology from \cite{IM-adapted}. If a pair of dilation parameters $\kappa = (\kappa_1, \kappa_2)$ 
is chosen so that $L_{\kappa} = \{ (t_1, t_2) : \kappa_1 t_1 + \kappa_2 t_2 = 1 \}$ is a supporting line
of the Newton polygon (that is, it contains a face $\tau = \tau_{\kappa}$ of the Newton diagram ${\mathcal N}_d(f)$), then
we call $f_{\tau}(x_1,x_2) = \sum_{(j,k) \in \tau} c_{j,k} x_1^j x_2^k$ the $\kappa$-principal part of $f$.
Abusing notation, we will sometimes denote this by $f_{\kappa}$. Note that
$f_{\kappa}(x_1, x_2)$ is a quasi-homogeneous polynomial such that 
$f_{\kappa}(r^{\kappa_1} x_1, r^{\kappa_2} x_2) = r f_{\kappa}(x_1, x_2)$.

\subsection{Prerequisites to the Algorithm}
The Weierstrass preparation theorem holds for the ring $\mathbb{Q}\{x_1,x_2\}$ of convergent power series
with rational coefficients. This can be seen by either modifying the proof of the Weierstrass preparation theorem for real coefficients given in \cite{Narasimhan} or observing that the Weierstrass preparation theorem holds for both $\mathbb{R}\{x_1, x_2\}$ (see \cite{Narasimhan}) and for the rings $\mathbb{Q}[[x_1, x_2]]$ of formal power 
series over $\mathbb{Q}$ and $\mathbb{R}[[x_1, x_2]]$ of formal power series over $\mathbb{R}$ (see \cite{Abhyankar}), and observing that the uniqueness of the factorisation in $\mathbb{R}[[x_1, x_2]]$ given by the Weierstrass preparation theorem implies that the factorisation over $\mathbb{Q}[[x_1, x_2]]$ and $\mathbb{R}\{x_1, x_2\}$ are the same. 

This means that given an analytic function $f \in \mathbb{Q}\{x,y\}$, convergent in a neighbourhood of the origin  (with the real topology on 
$\mathbb{Q}$), where $f(0, x_2) = x_1^{\nu_1} x_2^{\nu_2} f'(x_1, x_2)$, and where $f'(0, x_2) = x_2^m g(x_2)$, $g(0) \neq 0$, we can write $f$ in the form 
\[f(x_1, x_2) = U(x_1, x_2) x_1^{\nu_1} x_2^{\nu_2}F(x_1, x_2)\]
where
\[F(x_1, x_2) = x_2^m + g_1(x_1) x_2^{m-1} + \cdots + g_m(x_1)\]
where $U(0,0) \not= 0$ and $g_j(0)=0$ for all $j$. Furthermore
$g_1, \ldots, g_m \in \mathbb{Q}\{x_1\}$ are uniquely determined, not just in $\mathbb{Q}\{x_1\}$, but also as formal power series in the larger rings $\mathbb{Q}[[x_1]]$ and $\mathbb{R}[[x_1]]$. $U$ is also uniquely defined as a power series in $\mathbb{R}[[x_1, x_2]]$.



We may assume that $g_m$ is not zero. Then the roots $r(x_1)$ of $F(x_1, x_2)$ have a Puiseux series expansion
\[r(x_1) \ = \ c_1 x_1^{a} + c_2 x_1^{b} \ + \ \cdots \]
where, importantly for us, the nonzero {\it coefficients} $c_l$'s lie in $\mathbb{Q}^{\text{alg}}$, the algebraic closure of $\mathbb{Q}$
and the {\it exponents} $0< a < b < \cdots$ are a strictly increasing sequence of rational numbers. A reference showing the existence of a \emph{formal} Puiseux expansion of this form is Abhyankar's book \cite{Abhyankar}. Combining this with the usual Puiseux theorem for real power series as we did for the Weierstrass preparation theorem gives that the series describing each root is convergent.

The Puiseux expansion of two or more distinct roots $r$ of $F$ may agree for the first few terms and it will be important for us to
quantify this. 

We introduce the following notation from \cite{IM-adapted}. Let 
$a_1 < \cdots < a_n$ be the distinct leading exponents of the roots of $F$ so that each root
$r(x_1) = c x_1^{a_{l}} + O(x_1^{A})$ for some $c \not=0$,  $1\le l \le n$ and for some $A > a_{l}$.
For each $l \in \{1,2,\ldots, n\}$, we denote by $\begin{bsmallmatrix} \cdot \\ l \end{bsmallmatrix}$
the collection of roots with leading exponent $a_l$. Next, for every $1\le l \le n$, let $\{c_{l}^{(\alpha)}\}$ denote the collection of distinct, leading nonzero coefficients
appearing in the expansion of a root with leading exponent $a_l$ and let
$\begin{bsmallmatrix}\alpha \\ l\end{bsmallmatrix}$ denote the collection of roots with leading exponent $a_l$ and leading
coefficient $c_l^{(\alpha)}$. 

We continue to the second exponent in the expansion; for every $l_1$ and $\alpha_1$, we let
$\{a_{l_1, l}^{(\alpha_1)} :  l\ge 1 \}$ denote the collection of distinct exponents appearing in the second term of the Puiseux expansion
of the roots in 
$\begin{bsmallmatrix} \alpha_1 \\ l_1\ \end{bsmallmatrix}$. Proceeding in this way, we can express each
root $r$ as
\[r(x_1) = c_{l_1}^{(\alpha_1)} x_1^{a_{l_1}} + c_{l_1, l_2}^{(\alpha_1, \alpha_2)} x_1^{a_{l_1, l_2}^{(\alpha_1)}} + \cdots + c_{l_1, \ldots, l_p}^{(\alpha_1, \ldots, \alpha_p)} x_1^{a_{l_1, \ldots, l_p}^{(\alpha_1, \ldots, \alpha_{p-1})}} + \cdots\]
where the nonzero coefficients $c_l$ lie in $\mathbb{Q}^{\text{alg}}$ and
\[c_{l_1, \ldots, l_p}^{(\alpha_1, \ldots, \alpha_{p-1}, \beta)} \neq c_{l_1, \ldots, l_p}^{(\alpha_1, \ldots, \alpha_{p-1}, \gamma)}\]
whenever $\beta \neq \gamma$. Also
\[a_{l_1, \ldots, l_p}^{(\alpha_1, \ldots, \alpha_{p-1})} > \ a_{l_1, \ldots, l_{p-1}}^{(\alpha_1, \ldots, \alpha_{p-2})}\]
so that the terms in $r$ have increasing exponents. Furthermore the exponents are positive rational numbers.

The root cluster $\begin{bsmallmatrix} \alpha_1 & \cdots & \alpha_p \\ l_1 & \cdots & l_p \end{bsmallmatrix}$
denotes the collection of roots whose first $p$ leading terms are indexed by $l_1, \alpha_1, l_2, \alpha_2, \ldots, l_p, \alpha_p$.  We will also introduce clusters where the last \emph{exponent} has been picked but not the last \emph{coefficient}. These are denoted 
$\begin{bsmallmatrix} \alpha_1 & \cdots & \alpha_{p-1} & \cdot \\ l_1  & \cdots & \ l_{p-1} & \ l_p \end{bsmallmatrix}$
and equal the union over $\alpha_p$ of the clusters 
$\begin{bsmallmatrix} \alpha_1 & \cdots & \alpha_{p-1} & \alpha_p \\ l_1 & \cdots & \ l_{p-1} & l_p \end{bsmallmatrix}$.
The notation $N[\text{cluster}]$ will denote the number of roots in a cluster.

Since each $a_l$ corresponds to the cluster 
$\begin{bsmallmatrix} \cdot \\ l \end{bsmallmatrix}$,
the collection of roots of $F$ can be expressed as the union over all $l$ of these clusters. Then we can write
\[f(x_1, x_2) = U(x_1, x_2) x_1^{\nu_1} x_2^{\nu_2} \prod_{l=1}^n \Phi \begin{bsmallmatrix} \cdot \\ l \end{bsmallmatrix}
(x_1, x_2)\]
where 
\[\Phi \begin{bsmallmatrix}\cdot \\ l\end{bsmallmatrix}
(x_1, x_2) := \prod_{r \in \begin{bsmallmatrix}\cdot \\ l \end{bsmallmatrix}}
(x_2 - r(x_1)).\]

The advantage of this decomposition is that it allows us to read off the vertices of the Newton polygon.

\begin{lemma}\label{NewtonPolygonLemma}
The points $(A_l, B_l)$ 
where
\[A_l \ = \ \nu_1 + \sum_{\mu \leq l} a_{\mu} N \begin{bsmallmatrix} \cdot \\ \mu \end{bsmallmatrix} \ \ {\rm and} \ \ 
B_l \  = \ \nu_2 + \sum_{\mu \geq l+1}  N \begin{bsmallmatrix} \cdot \\\mu \end{bsmallmatrix} \]
 are the vertices of the Newton polygon of $f$. 
\end{lemma}
Here $l$ ranges between $0$ and $n$.
When $l=0$, we set $a_0 = 0$ so that 
$A_0 = \nu_1$ and $B_0 = \nu_2 + m$ where $m$ is the degree of $F(x_1, x_2)$ as a polynomial in $x_2$;
that is, the sum of the multiplicities of the roots of $F$. When $l=n$, $B_n = \nu_2$.
\begin{proof}
The Newton polygon of $f$ is the same as the Newton polygon of $x_1^{\nu_1} x_2^{\nu_2} F(x_1, x_2)$.

Consider any $\kappa >0$ not among the exponents $\{a_1, a_2, \ldots, a_n\}$ and choose $0 \le l_{\kappa} \le n$
so that $a_{l_{\kappa}} < \kappa < a_{l_{\kappa}+1}$ (if $a_n < \kappa$, choose $l_n$). Let 
$L_{\kappa} = \{ (t_1, t_2):  t_1 + \kappa t_2 = c_{\kappa} \}$ be a supporting line of the Newton polygon of $f$.
It either intersects the Newton diagram in a vertex or a compact edge as $\kappa$ is a positive, finite number.
In fact we will see that $L_{\kappa}$ intersects the Newton diagram in a vertex. 

We say that a monomial $x_1^c x_2^d$ in the Puiseux expansion of $F$ has degree $c + \kappa d$ with respect to the weight
$(1,\kappa)$. A necessary and sufficient condition for a point $(c_0,d_0)$ to lie on $L_{\kappa}$ is that it has minimal $(1,\kappa)$-degree among all the pairs $(c,d)$ arising as a monomial $x_1^c x_2^d$ in the Puiseux expansion of $F$.

For each factor $x_2 - r(x_1)$ arising in $F$ with the root $r(x_1)$ belonging to $\begin{bsmallmatrix}\cdot \\ l \end{bsmallmatrix}$,
the term $x_2$ has $(1,\kappa)$-degree equal to $\kappa$ and the minimal $(1,\kappa)$-degree among the terms in
the Puiseux expansion of $r(x_1)$ is $a_l$. Hence the lowest-degree $(1, \kappa)$-monomial appearing in $F$ is 
$x_1^{A_{l_{\kappa}}} x_2^{B_{l_{\kappa}}}$ since we take the $x_1^{a_{l}}$ term for 
$l \leq l_{\kappa}$ and the $x_2$ term for $l > l_{\kappa}$. This shows that $L_{\kappa}$ intersects the Newton
diagram at the vertex
$(A_{l_{\kappa}}, B_{l_{\kappa}})$.

Note that each $A_l$ must be an integer (it is obvious that $B_l$  is an integer) since the vertices of the Newton diagram
are lattice points. 
\end{proof}
Now, notice that $A_l - A_{l-1} = - a_l (B_l - B_{l-1})$, since $N \begin{bsmallmatrix} \cdot \\ l \end{bsmallmatrix}$ is equal to $B_l - B_{l-1}$.  From this it immediately follows that the slope of the line connecting $(A_{l-1}, B_{l-1})$ to $(A_l, B_l)$ is $-1 / a_l$. Therefore the
line connecting $(A_{l-1}, B_{l-1})$ to $(A_l, B_l)$ is given by $y = -(1/a_l) (x - A_l) + B_l$.
This line intersects the bisectrix at $(d_l, d_l)$ where $d_l = -(1/a_l) (d_l - A_l) + B_l$, so 
$d_l = \frac{A_l + a_l B_l}{1 + a_l}$.  If we index
this line $L_{\kappa^l} = \{(t_1, t_2): \kappa^l_1 t_1 + \kappa_2^l t_2 = 1 \}$ by the dilation parameters
$\kappa^l = (\kappa_1^l, \kappa_2^l)$, then 
\[ \kappa_1^l \ = \ \frac{1}{A_l + a_l B_l}, \ \ {\rm and} \ \  \kappa_2^l \ = \ \frac{a_l}{A_l + a_l B_l} \ \ {\rm so \ that} \ 
a_l \ = \ \frac{\kappa_2^l}{\kappa_1^l}. \]

The vertical edge, which passes through $(\nu_1, \nu_2 + m)$ (here $m$ is the sum of the multiplicities of all the roots $r(x_1)$
in $F$),
intersects the bisectrix at $(\nu_1, \nu_1)$, and the horizontal edge, passing through $(A_n, \nu_2)$, is contained in a line intersecting the bisectrix at $(\nu_2, \nu_2)$. So the distance $d(f)$ is given by $\max(\nu_1, \nu_2, \max_l d_l)$.

Finally, we observe that the $\kappa^l$-principal part of $f$ is the same as the $\kappa^l$-principal part of
\[c \, x_1^{\nu_1} x_2^{\nu_2} \prod_{j, \alpha} (x_2 - c_{j}^{(\alpha)} x_1^{a_{j}})^{N
\begin{bsmallmatrix} \alpha \\ j \end{bsmallmatrix}}\]
where $c = U(0,0)$. Since the $\kappa^l$-principal part of 
$x_2 - c_j^{(\alpha)} x_1^{a_j}$ equals $c_j^{(\alpha)} x_1^{a_j}$ if $j < l$
and equals $x_2$ if $l < j$, we have
\begin{equation}\label{fk}
f_{\kappa^{l}}(x_1, x_2) = c_l x_1^{A_{l-1}} x_2^{B_l} \prod_{\alpha} (x_2 - c_l^{(\alpha)} x_1^{a_l})^{N \begin{bsmallmatrix} \alpha \\ l \end{bsmallmatrix}}.
\end{equation}
In view of \eqref{fk} we say that the edge $[(A_{l-1}, B_{l-1}), (A_l, B_l)]$ is associated to the cluster
of roots $\begin{bsmallmatrix} \cdot \\ l \end{bsmallmatrix}$.
\subsection{The Algorithm}
We are now ready to describe the algorithm. Suppose that $f(x_1, x_2)$ is a real-analytic function
near $(0,0)$ with rational coefficients. Furthermore suppose that the coordinates $(x_1, x_2)$ are not adapted (otherwise
there is nothing to do).

We apply Theorem \ref{AdaptedCharacterization} part (a)
to conclude that the principal face $\pi(f)$ is a compact edge which lies
on a uniquely determined line $L_{\kappa} = \{ (t_1, t_2) : \kappa_1 t_1 + \kappa_2 t_2 = 1 \}$ with $\kappa_1, \kappa_2 > 0$.
The principal part $f_{\text pr}$ is just the $\kappa$-principal part of $f$.
By Lemma \ref{NewtonPolygonLemma} the compact edges of ${\mathcal N}_d(f)$ are given
by $[(A_{l-1}, B_{l-1}), \, (A_l, B_l)]$ with $1\le l \le n$. 
Choose $\lambda$ so that the principal face $\pi(f)$ of $f$
is $\tau_{\lambda} := [(A_{\lambda-1}, B_{\lambda-1}), \, (A_{\lambda}, B_{\lambda})]$. 
Therefore by \eqref{fk}, we have
\begin{equation}\label{fkpr}
f_{\text pr}(x_1, x_2) \ = \ f_{\kappa^{\lambda}}(x_1, x_2) \ = \ 
c \, x_1^{A_{\lambda-1}} x_2^{B_{\lambda}} \prod_{\alpha} (x_2 - c_{\lambda}^{(\alpha)} x_1^{a_{\lambda}})^{N \begin{bsmallmatrix} \alpha \\ l \end{bsmallmatrix}}.
\end{equation}
The slope
of $\tau_{\lambda}$ is $-1/a_{\lambda}$ so that $a_{\lambda} = \kappa_2/\kappa_1$. By
Theorem \ref{AdaptedCharacterization} part (b), $a_{\lambda} \in {\mathbb N}$. Furthermore
by part (c), there exists an index $\beta$ such that 
\begin{equation}\label{multi-d-inequality}
m_{\text pr}(f) \ = \  N \begin{bsmallmatrix} \beta \\ \lambda \end{bsmallmatrix} \ > \ d(f) \ = \ 
\frac{A_{\lambda} + a_{\lambda} B_{\lambda}}{1 + a_{\lambda}}, \ \ {\rm and} \ \  c_{\lambda}^{(\beta)} \in \mathbb{Q}.
\end{equation}
The root
$c_{\lambda}^{(\beta)}$ is the principal root of $f_{\text pr}$.

The first step is to apply $x = \sigma(y)$ where $y_1 := x_1$ and $y_2 := x_2 - c_{\lambda}^{(\beta)} x_1^{a_{\lambda}}$ and put $\tilde f = f \circ \sigma$. Then $\tilde f(y_1, y_2)$ is equal to $f(y_1, y_2 + c_{\lambda}^{(\beta)} y_1^{a_{\lambda}})$. Since $a_{\lambda}$ is an integer, this is a polynomial change of variables. 

We want to see what happens to the Newton diagram from this change of variables. 
We will use the \ $\widetilde{ }$ \ notation to denote quantities in the variables $(y_1, y_2)$; for example 
${\tilde U}(y_1, y_2) = U(y_1, y_2 + c_{\lambda}^{(\beta)} y_1^{a_{\lambda}})$. Hence
\[{\tilde f}(y_1, y_2) \ = \ {\tilde U}(y_1, y_2) y_1^{\nu_1} (y_2 +  c_{\lambda}^{(\beta)} y_1^{a_{\lambda}})^{\nu_2}
\prod_{l,\alpha} \bigl(y_2 - ( r(y_1) -  c_{\lambda}^{(\beta)} y_1^{a_{\lambda}})\bigr)^{N
\begin{bsmallmatrix} \alpha \\ l \end{bsmallmatrix}}\]
and so each root $\tilde r(y_1)$ of $\tilde f$ has the form
\[\tilde r(y_1) \ = \ c_l^{(\alpha_1)} y_1^{a_l} - c_{\lambda}^{(\beta)} y_1^{a_{\lambda}} \ + \ \text{higher order terms}.\]

For $l < \lambda$, the lowest degree term in the root is left unchanged, so we have $\tilde a_l = a_l$.  Furthermore the multiplicities $N \widetilde{\begin{bsmallmatrix} \cdot \\ l \end{bsmallmatrix}}$ are the same as the corresponding multiplicities for $f$. 

For $l > \lambda$, any root $r$ in any cluster $\begin{bsmallmatrix} \cdot \\ l \end{bsmallmatrix}$ (including the $x_2^{\nu_2}$ term) is transformed into a root with leading exponent $a_{\lambda}$. The same happens for roots in $\begin{bsmallmatrix} \cdot \\ \lambda \end{bsmallmatrix}$ that are not in $\begin{bsmallmatrix} \beta \\ \lambda \end{bsmallmatrix}$. Finally if
$r \in \begin{bsmallmatrix} \beta \\ \lambda \end{bsmallmatrix}$, then the leading exponent of ${\tilde r}$ is
of the form $a_{\lambda, l_2}^{(\beta)} > a_{\lambda}$. 

 Following Ikromov and M\"uller \cite{IM-adapted}, we separately consider two cases depending on whether or not there is a root that maps to a root with leading exponent $a_{\lambda}$. 

Case 1: This is the case where there is at least one root that maps to a root with leading exponent $a_{\lambda}$. 
This implies that $\tilde a_{\lambda} = a_{\lambda}$. We have $\tilde B_{\lambda} = N \begin{bsmallmatrix} \beta \\ \lambda \end{bsmallmatrix}$ since the roots ${\tilde r}$ with leading exponent greater than $a_{\lambda}$ are precisely
those roots corresponding to $r \in \begin{bsmallmatrix} \beta \\ \lambda \end{bsmallmatrix}$.

We then see that $\tilde A_{\lambda} = \tilde A_{\lambda -1} + a_{\lambda} B_{\lambda} - a_{\lambda} N \begin{bsmallmatrix} \beta \\ \lambda \end{bsmallmatrix}$ so that
\[(\tilde A_{\lambda}, \tilde B_{\lambda}) = \left( A_{\lambda-1} + a_{\lambda} B_{\lambda} - a_{\lambda} N \begin{bsmallmatrix} \beta \\ \lambda \end{bsmallmatrix}, N \begin{bsmallmatrix} \beta \\ \lambda \end{bsmallmatrix} \right).\]
The inequality in \eqref{multi-d-inequality} is equivalent to the statement that $\tilde A_{\lambda} < \tilde B_{\lambda}$. 
Therefore the edge $[(\tilde A_{\lambda - 1}, \tilde B_{\lambda - 1}), (\tilde A_{\lambda} \tilde B_{\lambda})]$ lies entirely above the bisectrix and is thus not the principal face. Hence the principal face is associated to some subcluster $\left[\begin{smallmatrix}\beta & \cdot \\ \lambda & \lambda_2 \end{smallmatrix} \right]$ in the original coordinates (or is a horizontal edge in which case the new coordinates are adapted).

Case 2: This is the other case. Now there is no root with leading exponent $a_{\lambda}$ in the new coordinates
and again the principal face corresponds to a subcluster of the same form (or is an unbounded edge in which case 
we are done).

If $\tilde f$ is not yet expressed in an adapted coordinate system (so that the conditions (1)-(3) in Theorems \ref{AdaptedCharacterization} still hold), we continue the procedure. Now, the later steps are similar. If the conditions (1)-(3) are satisfied, we again take the principal root, which is known to exist and is a rational number. In terms of the original coordinates, we now have a change of coordinates $x = \sigma_{(2)}(y)$ of the form
\[y_1 := x_1; \ \ y_2 := x_2 - (c_{\lambda}^{(\beta)} x_1^{a_{\lambda}} + c_{\lambda, \lambda_2}^{(\beta, \beta_2)} x_1^{a_{\lambda, \lambda_2}^{(\beta)}})\]
where the coefficients are, once again, rational, and the exponents are integers, and now the new principal face will be 
a compact edge associated to a further subcluster of the original root cluster, or it will be an unbounded edge, in which case the new coordinates are adapted.

We iterate this procedure. If this procedure terminates after finitely many steps, then we have arrived at a polynomial shear transformation that converts the coordinates into adapted coordinates. The conclusion of
Theorem \ref{CoV} therefore follows.

On the other hand, it is possible that this procedure does not terminate after finitely many steps. In this case, the multiplicities 
\[N_k := N \begin{bsmallmatrix}
\beta & \beta_2 & \cdots & \cdot \\ \lambda & \lambda_2 & \cdots & \lambda_{k+1}
\end{bsmallmatrix}\]
are a nonincreasing sequence of positive integers and hence eventually constant. We can therefore find a polynomial 
$\psi_0 \in {\mathbb Q}[X]$ such that the function $f_0(x_1, x_2) := f(x_1, x_2 + \psi_0(x_2))$  has an analytic root
\[\rho(x_1) := c_{\lambda}^{(\beta)} x_1^{a_{\lambda}} + \cdots\]
where each coefficient of this root is rational and where $\rho$ is not a polynomial. Furthermore, $\psi_0$ can be chosen so that $\rho(x_1)$ is the only root with leading exponent $\lambda$, but the root $\rho$ may have high multiplicity.

Now if we take $\tilde f(y_1, y_2) := f_0(y_1, y_2 + c_{\lambda}^{(\beta)} y_1^{a_{\lambda}})$, the previous arguments imply that the principal face of $\tilde f$ 
must be the final non-horizontal edge in the Newton diagram. Furthermore $\tilde f$ does not have a vanishing root because this would imply that $f_0$ has a root $c_{\lambda}^{(\beta)} x_1^{a_{\lambda}}$, which cannot exist because that would contradict the multiplicity assumption on $f$ and the particular choice of $\psi_0$. 

We claim that $\tilde f$ satisfies the conclusion of Theorem \ref{CoV}. We will do this by making a further, non-polynomial change of variables that yields an adapted coordinate system.

By the construction of ${\tilde f}$, the vertices of the Newton polyhedron of 
${\tilde f}$ are given by $(A_0, B_0), \ldots, (A_{\lambda} B_{\lambda})$ 
where $B_{\lambda} = 0$ and the principal edge is  
$[(A_{\lambda - 1}, B_{\lambda - 1}), (A_{\lambda}, B_{\lambda})]$, where $A_{\lambda -1} < B_{\lambda -1}$. 
From \eqref{fkpr}, we see that the principal part of ${\tilde f}$ is
\[\tilde f_{\text pr}(x_1, x_2) = c x_1^{A_{\lambda - 1}} (x_2 - c_{\lambda}^{(\beta)} x_1^{a_{\lambda}})^N\]
where $N = B_{\lambda - 1} > \nu_1$.

We will now apply Proposition \ref{h-homo} to
show that the height of $\tilde f_{\text pr}$ is equal to $N$. 
Since the principal face of ${\tilde f}$ is the compact edge 
$[(A_{\lambda - 1}, B_{\lambda - 1}), (A_{\lambda}, B_{\lambda})]$ and $A_{\lambda -1} < B_{\lambda -1} = N$,
we see that $d(f_{\text pr}) < N$.
But the root $c_{\lambda}^{(\beta)}$ has multiplicity $N$ as a root in the sense of the factorization \eqref{phi-factored}, so this must be the principal root of 
$\tilde f_{\text pr}$ and thus the height of ${\tilde{f_{\text pr}}}$ is $N$ by Proposition \ref{h-homo}.

We now consider the function $f^*(y_1, y_2)$ given by $\tilde f(y_1, y_2 + \rho(y_1))$. The nonzero roots $\tilde r$ are given by $r - \rho$ with $r \in \begin{bsmallmatrix} \cdot \\ l \end{bsmallmatrix}$ for some $l < \lambda$ and they have the same multiplicities and leading exponents as $r$. This change of variables deletes the last vertex of the Newton polygon 
since the last factor changes into $y_2^N$ and the principal face is now an unbounded horizontal edge. Therefore the Newton distance is $N$, the multiplicity of the vanishing root and so the height of $f$, the height of $f^*$, the height of $\tilde f$,
and the height of ${\tilde{f_{\text pr}}}$ are all equal to $N$.

The completes the proof of Proposition \ref{CoV}.

\section{Hensel's lemma and the proof of Proposition \ref{vc}}\label{hensel}

A weaker version of Proposition \ref{vc} was established in \cite{W-jga} but the argument given
in \cite{W-jga} readily extends to give a proof of Proposition \ref{vc}. Here we give an outline of the proof which
relies on a generalisation of the classical Hensel lemma. The following result was established in \cite{W-jga}.

\begin{lemma}\label{hensel-generalised} Let $g \in {\mathbb Z}_p[X]$ with $p > {\rm deg}(g)$. Suppose
there exists an integer $L \ge 1$ such that for any $x_0 \in {\mathbb Z}_p$,

1. \  $| g^{(k+1)}(x_0) g(x_0)| \ < \ | g^{(k)}(x_0) g'(x_0)|, \ \ {\rm for \ all} \ \ 1 \le k \le L-1$, and

2. \ $| g(x_0)| \ < \ |g^{(L)}(x_0) g'(x_0)|$.

Then there exists a unique $x \in {\mathbb Z}_p$ such that $g(x) = 0$ and $|x - x_0| \le |g(x_0) g'(x_0)^{-1} |$.
\end{lemma}

Remarks:

1. The lemma is valid for all primes $p$ but then the derivatives $g^{(k)}(x)$ appearing in the statement of the lemma
need to be replaced by $g^{(k)}(x)/k!$.

2. The $L=1$ case is the classical statement of Hensel's lemma. In this case, condition 1 is vacuous and 2 reduces to the usual hypothesis 
$|g(x_0)| < |g'(x_0)|^2$. In particular if $g(x_0) \equiv 0$ mod $p^s$ and $p^{\delta} || g'(x_0)$ where $\delta < s/2$,
then $|g(x_0)| < |g'(x_0)|^2$. The conclusion implies that there exists a unique $x \in {\mathbb Z}_p$ with
$x \equiv x_0$ mod $p^{s-\delta}$ and $g(x) = 0$.

3. The lemma holds in any field $K$, complete with respect to any nontrivial nonarchimedean absolute value $|\cdot|$
and $g \in {\mathfrak o}[X]$ where ${\mathfrak o} = \{ x \in K: |x| \le 1\}$.

4. The proof is a small variant of the usual proof of Hensel's lemma using the Newton formula to produce 
an approximating sequence to a solution of a polynomial equation.

We now turn to the proof of Proposition \ref{vc} where we  seek to prove the following: suppose $\psi \in {\mathbb Z}_p [X]$ and
that for some $n\ge 1$, $\psi^{(n)}(x_0)/n! \not\equiv 0$ mod $p$ for all $x_0 \in S$ in some set 
$S \subseteq {\mathbb Z}/p{\mathbb Z}$. Then for
$$ 
I \ := \ \sum_{x_0 \in S} \int_{B_{p^{-1}}(x_0)} {\rm e}(p^{-s} \psi(x)) \, dx, 
$$
we have $|I| \le C p^{-s/n}$ for all $s\ge 2$ with a constant $C$ depending only on $n$ and the degree of $\psi$.
This is the bound \eqref{vc-bound}. 

When $n=1$ then each integral in the above sum over $S$ vanishes. This follows in the same way we showed
${\mathcal I}_0 = 0$ in the proof of Theorem \ref{initial-bound}.

Suppose now $n\ge 2$, and, to simplify matters, we will assume that $s \equiv 0$ mod $n$. The other cases are slightly
more involved, especially the case $s \equiv 1$ mod $n$ but here we just want to give a general outline how to prove \eqref{vc-bound}.
When $s \equiv 0$ mod $n$, then $s = t n$ for some $t \ge 1$. We write
$$
I \ \ = \ \ \sum_{x_0 \in S} \sum_{\substack{u_0 \in \mathbb{Z}/p^{t} \mathbb{Z} \\ u_0 \equiv x_0 \text{ mod $p$}}}
\int_{B_{p^{-t}}(u_0)} {\rm e}(p^{-nt} \psi(x)) \, dx  \ \ =
$$
$$
\sum_{x_0 \in S} \sum_{\substack{u_0 \in \mathbb{Z}/p^{t} \mathbb{Z} \\ u_0 \equiv x_0 \text{ mod $p$}}} 
p^{-t} \int_{|u|\le 1} {\rm e}(p^{-nt} \psi(u_0 + p^t u)) \, du  =  
\sum_{x_0 \in S} \sum_{\substack{u_0 \in \mathbb{Z}/p^{t} \mathbb{Z} \\ u_0 \equiv x_0 \text{ mod $p$}}} 
 p^{-t} {\rm e}(p^{-nt} \psi(u_0)) T_{x_0, u_0}
$$
where
$$
T_{x_0, u_0} \ := \ \int_{|u|\le 1} {\rm e}
\bigl(p^{-(n-1)t} \sum_{r=1}^{n-1} \frac{1}{r!} \psi^{(r)}(u_0) p^{t(r-1)} u^r\bigr) du.
$$ 
We break up the sum over 
$$
R \ := \ \{(x_0, u_0) \in S \times {\mathbb Z}/p^t{\mathbb Z} : x_0 \equiv u_0 \, {\rm mod} \, p\} \ = \ R_1 \cup \cdots \cup R_n
$$
into $n$ disjoint sets where
$$
R_1 \ = \ \{ (x_0, u_0) \in R:  |\psi^{(n-1)}(u_0)| \le p^{-t} \},
$$
$$
R_2 \ = \ \{(x_0,u_0) \in R: |\psi^{(n-1)}(u_0)| > p^{-t} \ {\rm and} \ |\psi^{(n-2)}(u_0)| \le p^{-t} |\psi^{(n-1)}(u_0)| \}
$$
$$
\vdots
$$
$$
R_{n-1} \ = \ \Bigl\{(x_0, u_0) \in R : |\psi''(u_0)| > p^{-t} |\psi'''(u_0)| > \cdots > p^{-(n-2)t} 
$$
$$
{\rm and} \ \ |\psi'(u_0)| \le p^{-t} |\psi''(u_0)| \Bigr\},
$$
and
$$
R_n = \Bigl\{ (x_0, u_0) \in R: |\psi'(u_0)| > p^{-t} |\psi''(u_0)| > \cdots > p^{-(n-2)t}|\psi^{(n-1)}(u_0)| > p^{-(n-1)t} \Bigr\}.
$$

We make the following claim:
\begin{itemize}
\item  $\# R_j \le {\rm deg}(\psi), \ 1\le j\le n-1$; and
\item $T_{x_0, u_0} = 0 \ {\rm for \ every} \ (x_0, u_0) \in R_n$.
\end{itemize}

For $j=1$, we apply the classical Hensel lemma (the $L=1$ case in Lemma \ref{hensel-generalised}) to $g(x) = \psi^{(n-1)}(x)$
to deduce that for every $(x_0,u_0) \in R_1$, there exists a unique $x \in {\mathbb Z}_p$ such that $\psi^{(n-1)}(x) = 0$
and $x \equiv u_0$ mod $p^t$. Hence $\# R_1 \le {\rm deg}(g) \le {\rm deg}(\psi)$.

Next for $(x_0, u_0) \in R_2$, consider $g(x) = \psi^{(n-2)}(x)$ so that $|g(u_0)| \le p^{-t} |g'(u_0)|$ and $|g'(u_0)| > p^{-t}$.
Once again the classical version of Hensel implies that there exists a unique $x\in {\mathbb Z}_p$ such that $g(x) = 0$
and $x \equiv u_0$ mod $p^t$. Hence $\# R_2 \le {\rm deg}(g) \le {\rm deg}(\psi)$.

Now for $(x_0, u_0) \in R_j$ with $3\le j \le n-1$, consider $g(x) = \psi^{(n-j)}(x)$ so that $|g(u_0)| \le p^{-t}|g'(u_0)|$ and
$|g'(u_0)| > p^{-t} |g''(u_0)| > \cdots > p^{ - (j-1) t}$. Applying Lemma \ref{hensel-generalised} with $L=j-1$ shows that
there exists a unique $x\in {\mathbb Z}_p$ with $g(x) = 0$ and $x \equiv u_0$ mod $p^t$. Hence $\# R_j \le {\rm deg}(g) \le
{\rm deg}(\psi)$. 

Finally for $(x_0, u_0) \in R_n$, we define $\sigma = t(n-1) - t - \nu$ where $p^{-\nu} := |\psi''(u_0)|$. Note that
$(x_0,u_0) \in R_n$ implies that $|\psi''(u_0)| > p^{-(n-1)t + t}$ and so $t + \nu < (n-1) t$, implying $\sigma \ge 1$. Hence,
setting 
$$
\Psi(u) \ := \ \sum_{r=1}^{n-1} \frac{1}{r!} \psi^{(r)}(u_0) p^{t(r-1)} u^r,
$$
we have
$$
T_{x_0, u_0} \ = \ \int_{|u|\le 1} {\rm e}(p^{-(n-1)t} \Psi(u)) du \ = \ \sum_{w \in {\mathbb Z}/p^{\sigma}{\mathbb Z}}
\int_{B_{p^{-\sigma}}(w)} {\rm e}(p^{-(n-1)t} \Psi(u)) du
$$ 
$$
= \ \sum_{w \in {\mathbb Z}/p^{\sigma}{\mathbb Z}} p^{-\sigma} \int_{|y|\le 1} {\rm e}(p^{-(n-1)t} \Psi(w + p^{\sigma} y)) \, dy.
$$
Now observe that
$\Psi(w + p^{\sigma} y) = \Psi(w) + p^{\sigma} \psi'(u_0) y \ +$
$$  
\frac{1}{2} \psi''(u_0) p^t \bigl((w + p^{\sigma} y)^2 - w^2\bigr) \ + \cdots + \
\frac{1}{(n-1)!} \psi^{(n-1)}(u_0) p^{t(n-2)} \bigl( (w + p^{\sigma} y)^{n-1} - w^{n-1}\bigr).
$$
However since $(x_0, u_0) \in R_n$,
\[\left|\frac{1}{2} \psi''(u_0) p^t \bigl((w + p^{\sigma} y)^2 - w^2\bigr)\right|  =  p^{-t - \sigma - \nu} =\ p^{-t(n-1)} \]
and, by comparing $\psi^{(j)}(u_0)$ to $\psi''(u_0)$ and using the fact that $\sigma > 1$, we have for $3 \leq j \leq n-1$:
\begin{IEEEeqnarray*}{Cl}
& \left| \frac{1}{j!} \psi^{(j)}(u_0) p^{t(j-1)} \bigl( (w + p^{\sigma} y)^{j} - w^{j}\bigr) \right|\\
\leq & p^{(j-2)t} \left| \psi''(u_0) p^{t(j-1)} \bigl( (w + p^{\sigma} y)^{j} - w^{j}\bigr) \right|\\
\leq & \left|\psi''(u_0) p^t p^{\sigma} \right|\\
=    & p^{-t(n-1)}.
\end{IEEEeqnarray*}
This means that the $j \geq 2$ terms in the sum defining $\Psi$ are divisible by $p^{t(n-1)}$.
Hence
$$
T_{x_0, u_0} \ = \  \sum_{w \in {\mathbb Z}/p^{\sigma}{\mathbb Z}} p^{-\sigma} {\rm e}(p^{-(n-1)t} \Psi(w))
\int_{|y|\le 1} {\rm e}(p^{-(n-1)t + \sigma} \psi'(u_0) y) \, dy
$$
and this last integral is equal to zero since $(x_0, u_0) \in R_n$ implies
$$
|\psi'(u_0)| > p^{-t} |\psi''(u_0)| \ = \ p^{-t - \nu} \ = \ p^{\sigma - t(n-1)}
$$
and so $p^{t(n-1) - \sigma} \not| \, \psi'(u_0)$. 

This establishes the claim which implies
$$
|I| \ \le \ \Bigl| p^{-t} \sum_{j=1}^{n-1} \sum_{(x_0, u_0) \in R_j} {\rm e}(p^{-nt} \psi(u_0)) T_{x_0, u_0} \Bigr| \ \le \
(n-1) {\rm deg}(\psi) p^{-t} \ = \  C \, p^{-s/n},
$$
giving us \eqref{vc-bound}.

\bibliographystyle{plain}
\bibliography{LocalSum-paper}

\end{document}